\documentclass[11pt]{amsart}
\usepackage{amsmath}
\usepackage{amssymb}
\usepackage{graphicx}
\usepackage{url}
\usepackage{multirow}

\usepackage{algorithm}
\usepackage{algorithmic}
\usepackage{float}
\usepackage{comment}
\usepackage[font=small,labelfont=bf]{caption}
\usepackage{subcaption}

\usepackage{color}
\usepackage{hyperref}
\usepackage{multicol}

\newtheorem{theorem}{Theorem}[section]
\newtheorem{corollary}[theorem]{Corollary}
\newtheorem{lemma}[theorem]{Lemma}

\theoremstyle{definition}

\newtheorem{remark}[theorem]{Remark}
\newtheorem{assumption}[theorem]{Assumption}

\numberwithin{equation}{section}
\title[Optimal Convergence Rate] {Optimal Convergence Rate for Mirror Descent Methods with special Time-Varying Step Sizes Rules}

\author[M.~S.~Alkousa]{Mohammad S. Alkousa}
\address[M.~S.~Alkousa]{Moscow Institute of Physics and Technology, Russia.}
\email{\tt mohammad.alkousa@phystech.edu}

\author[F.~S.~Stonyakin]{Fedor~S.~Stonyakin}
\address[F.~S.~Stonyakin]{Moscow Institute of Physics and Technology and V. I. Vernadsky Crimean Federal University, Russia.}
\email{{\tt fedyor@mail.ru}}

\author[A.~M.~Abdo]{Asmaa~M.~Abdo}
\address[A.~M.~Abdo]{Mathematics Department, Faculty of Science, Damascus University, Damascus, Syria.}
\email{{\tt asmaa.abdo@damascusuniversity.edu.sy}}

\author[M.~M.~Alcheikh]{Mohammad~M.~Alcheikh}
\address[M.~M.~Alcheikh]{Mathematics Department, Faculty of Science, Damascus University, Damascus, Syria.}
\email{{\tt 	mohammad.alcheikh@damascusuniversity.edu.sy}}


\begin{document}

\begin{abstract}
In this paper, the optimal convergence rate $O\left(N^{-1/2}\right)$ (where $N$ is the total number of iterations performed by the algorithm), without the presence of a logarithmic factor, is proved for mirror descent algorithms with special time-varying step sizes, for solving classical constrained non-smooth problems, problems with the composite model {\color{black} and problems with non-smooth functional (inequality types) constraints}. The proven result is an improvement on the well-known rate $O\left(N^{-1/2} \log (N) \right)$ for the mirror descent algorithms with the time-varying step sizes under consideration.  It was studied a new weighting scheme assigns smaller weights to the initial points and larger weights to the most recent points. This scheme improves the convergence rate of the considered mirror descent methods, which in the conducted numerical experiments outperform the other methods providing a better solution in all the considered test problems.
\end{abstract}

\maketitle

\section{Introduction}

The first method for unconstrained minimization of non-smooth convex function was proposed with fixed step size in \cite{shor_book}. In the next years, there were developed several strategies for choosing the step sizes \cite{polyak_book}. Among them, are fixed step length, non-summable diminishing step, square summable but not summable step \cite{Boyd2004Subgradient}, Polyak step \cite{polyak_book}, and many others (see Table \ref{Tab_steps}, below). The optimal convergence rate of the subgradient method is known as $O\left(N^{-1/2}\right)$, where $N$ is the total number of iterations performed by the algorithm \cite{Beck2017book,Bubeck_book,Lan2020book,Nesterov_book}. This rate was proved by using many rules of the step size. One of them is fixed (constant), namely $\gamma_k = O(1)/\sqrt{N} $ for $ k = 1, \ldots, N$ (with special chosen of the constant $O(1)$), which minimizes the upper bound of the difference between the best value of the objective function attained (i.e., the best value observed during the minimization process) and the optimal value. But most of the time it is preferable to use the time-varying step size since the constant step size rule requires the a priori knowledge of the total number of iterations employed by the method. From the mentioned fixed steps, the inspiration was to take $\gamma_k = O(1)/\sqrt{k}$ for $k = 1, \ldots, N$, but the convergence rate, with this rule,  became sub-optimal $O\left(N^{-1/2} \log (N) \right)$. Recently in \cite{Zhihan2023Convergence}, for the projected subgradient method it was proved the optimal convergence rate using the previously mentioned time-varying step size with a new weighting scheme for the generated points each iteration of the algorithm. 
 
The next important advancement in this area is related to the development of the mirror descent method which originated in \cite{Nemirovskii1979efficient,Nemirovsky1983Complexity} and was later analyzed in \cite{Beck2003Mirror}. It is considered as the non-Euclidean extension of standard subgradient methods. This method is used in many applications, see \cite{applications_tomography_2001,article:Nazin_Recursive_Aggregation,article:Nazin_2011,article:Nazin_2014,article:Nazin_2013} and references therein.  The standard subgradient methods employ the Euclidean distance function with a suitable step size in the projection step. Mirror descent extends the standard projected subgradient methods by employing a nonlinear distance function with an optimal step size in the nonlinear projection step \cite{article:luong_weighted_mirror_2016}. %The Mirror Descent method not only generalizes the standard subgradient methods but also achieves a better convergence rate \cite{article:doan_2019}. 
The mirror descent method is also applicable to optimization problems in Banach spaces where gradient descent is not \cite{article:doan_2019}. An extension of the mirror descent method for constrained problems was proposed in \cite{article:beck_comirror_2010,Nemirovsky1983Complexity}. 
Usually, the step size and stopping criterion for the mirror descent method require knowing the Lipschitz constant of the objective function and constraints, if any. Adaptive step sizes, which do not require this information, are considered for unconstrained problems in \cite{book:nemirovski_lectures_2001}, and for constrained problems in \cite{article:beck_comirror_2010}. 

{\color{black} In addition to the previously mentioned problems, there is a more general class of problems, which is non-smooth optimization problems with non-smooth functional (inequality types) constraints (see problem  \eqref{general_problem_1} or its equivalence \eqref{general_problem}). This class of problems arises and attracts widespread interest in many areas of modern large-scale optimization and its applications \cite{Nemirovski_Robust,Shpirko_primal}. On continuous optimization with functional constraints, there is a long history of studies. The monographs in this area include \cite{24,94}. Some of the works on first-order methods for convex optimization problems with convex functional constraints include \cite{article:adaptive_mirror_2018,43,71,stonyakin2018,stonyakin2019,119,125} for the deterministic setting and \cite{3,2,66,126} for the stochastic setting.}

As we pointed out the mirror descent algorithm was proposed and analyzed a long time ago for different step size rules. However, the optimal convergence rate $O(N^{-1/2})$ of the algorithm was first proved, as for the subgradient method,  with the fixed step size $\gamma_k = O(1)/\sqrt{N} $ for $ k = 1, \ldots, N$. But, when the step sizes are time-varying, namely $\gamma_k = O(1)/\sqrt{k} $,  it was proved the sub-optimal convergence rate $O(N^{-1/2} \log(N) )$, see \cite{Beck2017book,Bubeck_book,Lan2020book,Nesterov_book} for details. Although the constant step size rule is relatively easy to analyze it has the disadvantage of requiring the a priori knowledge of the total number of iterations employed by the method. In practical situations, the number of iterations is not fixed a priori, and stopping criteria different than merely fixing the total number of iterations is usually imposed. This is why dynamic, or time-varying (i.e., non-constant) step size rules are important \cite{Bubeck_book}. 

%Recently in \cite{Zhihan2023Convergence}, it was proved the optimal convergence rate for the classical projected subgradient method for time-varying step size.   This convergence rate remains the same (optimal) even if we slightly increase the weight of the most recent points, thereby relaxing the ergodic sense.

In this paper, we extend the results of \cite{Zhihan2023Convergence} to the mirror descent methods.  We proved the optimal convergence rate, without the presence of a logarithmic factor, for mirror descent methods with special time-varying step sizes, for solving classical constrained non-smooth problems (see Sect. \ref{sectinMirror}), constrained composite problems (see Sect. \ref{sectinMirrorC}), and problems with functional constraints (see Sect. \ref{section_functio_const}). We demonstrate the efficiency of the theoretical studied results by some conducted experiments for some constrained (with and without functional constraints) optimization problems, such as best approximation problem, Fermat-Torricelli-Steiner problem, smallest covering ball problem and the maximum of a finite collection of linear functions.

The paper consists of an introduction and 5 main sections. In Sect. \ref{sect_basics}  we mentioned the basic facts and tools and mirror descent basics.  Sect. \ref{sectinMirror} devoted to the proof of the optimal convergence rate of the mirror descent algorithm with time-varying step size rules for solving the classical constrained non-smooth optimization problem. In Sect. \ref{sectinMirrorC}, the same analysis that in Sect. \ref{sectinMirror},  performed for more general optimization problems, namely the constrained composite optimization problem. {\color{black} Section \ref{section_functio_const} is devoted to the problem with functional constraints, which is a more general class of non-smooth optimization problems.}  In Sect. \ref{sect_numerical}, we present some numerical experiments which demonstrate the effectiveness of the proposed algorithms and analysis.

\section{Fundamentals }\label{sect_basics}

Let $(\mathbf{E},\|\cdot\|)$ be a normed finite-dimensional vector space, with an arbitrary norm $\|\cdot\|$, and $\mathbf{E}^*$ be the conjugate space of $\mathbf{E}$ with the following norm
$$
\|y\|_{*}=\max\limits_{x \in \mathbf{E}}\{\langle y,x\rangle: \|x\|\leq1\},
$$
where $\langle y,x\rangle$ is the value of the continuous linear functional $y \in \mathbf{E}^*$ at $x \in \mathbf{E}$.

Let $Q \subset \mathbf{E}^n$ be a {\color{black} compact} convex set, and %contained in an Euclidean ball with radius $R > 0$.  
$\psi: Q \longrightarrow \mathbb{R}$ be a proper closed differentiable and $\sigma$-strongly convex (called prox-function or distance generating function) with $\sigma > 0 $. The corresponding Bregman divergence is defined as 
$$
V_{\psi} (x, y) = \psi (x) - \psi (y) - \langle \nabla \psi (y), x - y \rangle, \quad \forall x , y \in Q. 
$$
For the Bregmann divergence, it holds the following inequality
\begin{equation}\label{eq_breg}
V_{\psi} (x, y) \geq \frac{\sigma}{2} \|y - x\|^2, \quad \forall x, y \in Q. 
\end{equation}

In what follows, we denote the subdifferential of $f$ at $x$ by $\partial f(x)$, and the subgradient of $f$ at any point $x$ by $\nabla f(x) \in \partial f(x)$.   Let $\operatorname{dom} (f)$ denote the domain of the function $f$, and $\operatorname{dom} (\partial f)$ denote the set of points of subdifferentiability of $f$, i.e., 
$$
\operatorname{dom} (\partial f) = \left\{ x \in  \mathbf{E} : \partial f(x) \ne \emptyset  \right\}. 
$$

The following identity, known as the three points identity, is essential in the analysis of the mirror descent method.

\begin{lemma}\label{three_points_lemma}(Three points identity) \cite{Chen1993}
Suppose that $\psi: \mathbf{E} \longrightarrow (- \infty, \infty]$ is proper
closed, convex, and differentiable over $\operatorname{dom}(\partial \psi)$. Let $a, b \in \operatorname{dom}(\partial(\psi))$ and $c \in \operatorname{dom} (\psi)$. Then it holds %the following equality
\begin{equation}\label{eq_three_points}
\langle\nabla \psi(b)-\nabla \psi(a), c-a\rangle=V_{\psi}(c,a)+V_{\psi}(a, b)-V_{\psi}(c, b) .
\end{equation}

\end{lemma}

\noindent
\textbf{Fenchel-Young inequality.} For any $a \in \mathbf{E}, b \in \mathbf{E}^*$, it holds the following inequality
\begin{equation}\label{Fenchel_Young_ineq}
  \langle a,b \rangle \leq \frac{\|a\|^2}{2 \lambda} + \frac{\lambda\|b\|_*^2}{2}, \quad \forall \lambda > 0 .
\end{equation}

%In what follows, we denote the subdifferential of $f$ at $x$ by $\partial f(x)$, and the subgradient of $f$ at any point $x$ by $\nabla f(x) \in \partial f(x)$. 

%%%%%%%%%%%%%%%%%%%%%%%%%%%%%%%%%%%%%%%%%%%%%%%%%%%%%%%%%%
\section{Mirror descent method for constrained problem}\label{sectinMirror}

%see \url{https://www.pokutta.com/blog/research/2019/02/27/cheatsheet-nonsmooth.html} 

In this section, we consider the following constrained optimization problem
\begin{equation}\label{main_constrained_prob}
\min_{x \in Q} f(x),
\end{equation}
where $f$ is a convex Lipschitz continuous function (non-smooth)  with Lipschitz constant $M_f >0$, i.e., 
\begin{equation}\label{lipschitz_cond}
|f(x) - f(y)| \leq M_f \|x - y\|, \quad \forall x, y \in Q. 
\end{equation}

{\color{black}For problem \eqref{main_constrained_prob}, we use Algorithm \ref{alg_mirror_descent}, under consideration
\begin{equation}\label{cond_x1xstar}
\max_{x \in Q} V_{\psi} (x^*, x) \leq V_{\psi} (x^*, x^1) < \infty,    
\end{equation}
where $x^1 \in Q$ is a chosen (dependently on $Q$) initial point and $x^*$ is an optimal solution of \eqref{main_constrained_prob} nearest to $x^1$.}

\begin{algorithm}[!ht]
\caption{Mirror Descent Method.}\label{alg_mirror_descent}
\begin{algorithmic}[1]
\REQUIRE step sizes $\{\gamma_k\}_{k \geq 1}$,  initial point $x^1  \in Q$ {\color{black} s.t. \eqref{cond_x1xstar} holds}, number of iterations $N$.
\FOR{$k= 1, 2, \ldots, N$}
\STATE Calculate $\nabla f(x^k) \in \partial f(x^k)$,
\STATE $x^{k+1} = \arg\min_{x \in Q} \left\{ \langle x, \nabla f(x^k)  \rangle + \frac{1}{\gamma_k} V_{\psi}(x,x^k) \right\} $.
\ENDFOR
%\ENSURE $\hat{x} = \frac{1}{k} \sum_{i = 1}^{k} x^i$.
\end{algorithmic}
\end{algorithm}

\begin{theorem}\label{theo_main_ineq_mirror_desc}
Let $f$ be an $M_f$-Lipschitz convex function, and assume $ \|\nabla f(x)\|_* \leq M_f, \forall x \in Q$. Then for problem \eqref{main_constrained_prob}, by Algorithm \ref{alg_mirror_descent}, with a positive non-increasing sequence of step sizes $\{\gamma_k\}_{k \geq 1}$, for any fixed $m \geq -1$, it satisfies the following inequality
\begin{equation}\label{main_ineq_mirror_desc}
f(\hat{x}) - f(x^*) \leq  \frac{1}{\sum_{k= 1}^{N} \gamma_k^{-m}} \left( \frac{V_{\psi}(x^*, x^1) }{\gamma_N^{m+1}} + \frac{1}{2 \sigma} \sum_{k= 1}^{N} \frac{\|\nabla f(x^k)\|_*^2}{\gamma_k^{m -1}} \right),
\end{equation}
where $\hat{x} = \frac{1}{\sum_{k= 1}^{N} \gamma_k^{-m}} \sum_{k = 1}^{N} \gamma_k^{-m} x^k $, and $x^*$ is an optimal solution of \eqref{main_constrained_prob}.
\end{theorem}

\begin{proof}
Let $\tilde{f}(x) :=  \langle x, \nabla f(x^k)  \rangle + \frac{1}{\gamma_k} V_{\psi}(x,x^k)$, then from Algorithm \ref{alg_mirror_descent} we have $x^{k+1} =\arg\min\limits_{x \in Q} \tilde{f}(x) $. By the optimality condition, we get
$$
\langle \nabla \tilde{f}(x^{k+1}), x - x^{k+1} \rangle \geq 0, \quad \forall x \in Q.
$$
Thus, 
$$
\langle \gamma_k \nabla f(x^k) + \nabla \psi(x^{k+1})- \nabla \psi(x^k), x - x^{k+1} \rangle \geq 0, \quad \forall x \in Q.
$$
In particular for $x = x^*$, we have
\begin{equation}\label{eq_1}
\langle \nabla \psi(x^k) -\nabla \psi(x^{k+1}) - \gamma_k \nabla f(x^k), x^* - x^{k+1} \rangle \leq 0.
\end{equation}

By using the definition of a subgradient of $f$ at $x^*$, we get
\begin{equation*}
   \begin{aligned}
       \gamma_k \left( f(x^k) - f(x^*) \right) & \leq  \gamma_k \langle \nabla f(x^k), x^k - x^* \rangle
       \\& = \underbrace{\langle \nabla \psi(x^k) -\nabla \psi(x^{k+1}) - \gamma_k \nabla f(x^k), x^* - x^{k+1} \rangle}_{:=J_1} 
       \\& \;\;\;\; + \underbrace{\langle \nabla \psi(x^{k+1}) -\nabla \psi(x^{k}), x^* - x^{k+1}  \rangle}_{: = J_2} 
        \\& \;\;\;\; + \underbrace{\langle \gamma_k \nabla f(x^{k}), x^k - x^{k+1}  \rangle}_{: = J_3} . 
   \end{aligned} 
\end{equation*}
From \eqref{eq_1}, we find $J_1 \leq 0$. By Lemma \ref{three_points_lemma}, we have 
\begin{equation*}
   \begin{aligned}
       J_2 & = - \langle \nabla \psi(x^{k}) -\nabla \psi(x^{k+1}), x^* - x^{k+1}  \rangle
        \\& = - \left( V_{\psi} (x^*, x^{k+1}) + V_{\psi} (x^{k+1}, x^k) - V_{\psi} (x^*, x^{k}) \right) 
       \\& = V_{\psi} (x^*, x^{k}) - V_{\psi} (x^*, x^{k+1}) - V_{\psi} (x^{k+1}, x^k). 
   \end{aligned} 
\end{equation*}
By the Fenchel-Young inequality \eqref{Fenchel_Young_ineq}, with $\lambda = \sigma > 0 $, we find 
$$
J_3 \leq \frac{\gamma_k^2}{2 \sigma} \|\nabla f(x^k)\|_*^2 + \frac{\sigma}{2} \|x^k - x^{k+1}\|^2. 
$$
Therefore, we get
$$
\begin{aligned}
\gamma_k \left( f(x^k) - f(x^*) \right) & \leq V_{\psi} (x^*, x^{k}) - V_{\psi} (x^*, x^{k+1}) - V_{\psi} (x^{k+1}, x^k) 
\\& \;\;\;\;+ \frac{\gamma_k^2}{2 \sigma} \|\nabla f(x^k)\|_*^2  + \frac{\sigma}{2} \|x^k - x^{k+1}\|^2. 
\end{aligned}
$$

From \eqref{eq_breg}, we have %the $\sigma$-strong convexity of $V_{\psi} (\cdot, \cdot)$, we have
$$
V_{\psi} (x^{k+1}, x^{k}) \geq \frac{\sigma}{2} \|x^{k+1} - x^k\|^2 .
$$

Thus, we get the following inequality
\begin{equation}\label{eq_2}
f(x^k) - f(x^*) \leq \frac{1}{\gamma_k} \left( V_{\psi} (x^*, x^{k}) - V_{\psi} (x^*, x^{k+1}) \right) + \frac{\gamma_k}{2 \sigma} \|\nabla f(x^k)\|_*^2. 
\end{equation}

By multiplying both sides of \eqref{eq_2} by $\frac{1}{\gamma_k^m}$, and taking the sum from $1$ to $N$, we get
$$
\begin{aligned}
    \sum_{k = 1}^{N} \frac{1}{\gamma_k^m} \left( f(x^k) - f(x^*) \right) & \leq  \sum_{k = 1}^{N}\frac{1}{\gamma_k^{m+1}} \left( V_{\psi} (x^*, x^{k}) - V_{\psi} (x^*, x^{k+1}) \right) 
    \\& \;\;\;\; + \sum_{k = 1}^{N} \frac{\|\nabla f(x^k)\|_*^2}{2 \sigma \gamma_k^{m-1}}.
\end{aligned}
$$

Since $f$ is a convex function, we have 
$$
\begin{aligned}
    \left(\sum_{k = 1}^{N} \frac{1}{\gamma_k^{m}}\right) & \left[ f\left(\frac{1}{\sum_{k = 1}^{N} \gamma_k^{-m}} \sum_{k = 1}^{N} \gamma_k^{-m} x^k\right) - f(x^*) \right] \leq
    \\& \qquad \qquad \qquad \qquad \qquad \leq \sum_{k = 1}^{N} \frac{1}{\gamma_k^m} \left( f(x^k) - f(x^*) \right). 
\end{aligned}
$$

Therefore, by setting $\hat{x} : = \frac{1}{\sum_{k = 1}^{N} \gamma_k^{-m}} \sum_{k = 1}^{N} \gamma_k^{-m} x^k$, we get
\begin{equation*}
\begin{aligned}
    & \left(\sum_{k = 1}^{N} \gamma_k^{-m}\right) \left(f(\hat{x}) - f(x^*)\right)  %%% OK
     \leq
    \frac{1}{\gamma_1^{m+1}} \left( V_{\psi}(x^*, x^1)  - V_{\psi}(x^*, x^2)\right)   
    \\& \;\;\;\; + \sum_{k = 2}^{N-1}\frac{1}{\gamma_k^{m+1}} \left( V_{\psi}(x^*, x^k)  - V_{\psi}(x^*, x^{k+1})\right)
    \\& \;\;\;\; + \frac{1}{\gamma_N^{m+1}} \left( V_{\psi}(x^*, x^N)  - V_{\psi}(x^*, x^{N+1})\right) 
    + \frac{1}{2 \sigma} \sum_{k = 1}^{N} \frac{\|\nabla f(x^k)\|_*^2}{\gamma_k^{m-1}}   %%% OK
    \\& \leq \frac{1}{\gamma_1^{m+1}} V_{\psi}(x^*, x^1) +   \frac{1}{\gamma_N^{m+1}} V_{\psi}(x^*, x^N) + \sum_{k = 2}^{N-1} \frac{1}{\gamma_k^{m+1}} V_{\psi}(x^*, x^k)  
    \\& \;\;\;\;- \frac{1}{\gamma_1^{m+1}} V_{\psi}(x^*, x^2) - \sum_{k = 2}^{N-1} \frac{1}{\gamma_k^{m+1}} V_{\psi}(x^*, x^{k+1}) 
    + \frac{1}{2 \sigma} \sum_{k = 1}^{N} \frac{\|\nabla f(x^k)\|_*^2}{\gamma_k^{m-1}}   %%% OK=========
    \\& = \frac{1}{\gamma_1^{m+1}} V_{\psi}(x^*, x^1) + \sum_{k = 2}^{N} \frac{1}{\gamma_k^{m+1}} V_{\psi}(x^*, x^k) - \frac{1}{\gamma_1^{m+1}}   V_{\psi}(x^*, x^2) 
    \\& \;\;\;\; - \sum_{k = 3}^{N} \frac{1}{\gamma_{k-1}^{m+1}} V_{\psi}(x^*, x^k)
    + \frac{1}{2 \sigma} \sum_{k = 1}^{N} \frac{\|\nabla f(x^k)\|_*^2}{\gamma_k^{m-1}} %%% OK==========
    \\& = \frac{1}{\gamma_1^{m+1}} V_{\psi}(x^*, x^1) + \sum_{k = 2}^{N} \frac{1}{\gamma_k^{m+1}} V_{\psi}(x^*, x^k)  - \sum_{k = 2}^{N} \frac{1}{\gamma_{k-1}^{m+1}} V_{\psi}(x^*, x^k)
\\& \;\;\;\;   + \frac{1}{2 \sigma} \sum_{k = 1}^{N} \frac{\|\nabla f(x^k)\|_*^2}{\gamma_k^{m-1}}
\end{aligned}    
\end{equation*}

Due to \eqref{cond_x1xstar}, we find that $V_{\psi} (x^*, x^k) \leq V_{\psi} (x^*, x^1), \, \forall k = 2, \ldots, N$. Hence we get
\begin{equation*}
\begin{aligned}
& \left(\sum_{k = 1}^{N} \gamma_k^{-m}\right) \left(f(\hat{x}) - f(x^*)\right) \leq   %\frac{1}{\gamma_1^{m+1}} V_{\psi}(x^*, x^1) + \sum_{k = 2}^{N} \frac{1}{\gamma_k^{m+1}} V_{\psi}(x^*, x^k) 
%\\& \;\;\;\;  - \sum_{k = 2}^{N} \frac{1}{\gamma_{k-1}^{m+1}} V_{\psi}(x^*, x^k) + \frac{1}{2 \sigma} \sum_{k = 1}^{N} \frac{\|\nabla f(x^k)\|_*^2}{\gamma_k^{m-1}} %%% OK==========
\\&  \leq \frac{1}{\gamma_1^{m+1}} V_{\psi}(x^*, x^1) +V_{\psi}(x^*, x^1)  \sum_{k = 2}^{N} \left(\frac{1}{\gamma_k^{m+1}} - \frac{1}{\gamma_{k-1}^{m+1}}\right) 
+ \frac{1}{2 \sigma} \sum_{k = 1}^{N} \frac{\|\nabla f(x^k)\|_*^2}{\gamma_k^{m-1}} %%% OK==========
\\& = \frac{1}{\gamma_1^{m+1}} V_{\psi}(x^*, x^1) + V_{\psi}(x^*, x^1) \left(-\frac{1}{\gamma_1^{m+1}} + \frac{1}{\gamma_{N}^{m+1}} \right) 
+ \frac{1}{2 \sigma} \sum_{k = 1}^{N} \frac{\|\nabla f(x^k)\|_*^2}{\gamma_k^{m-1}} %%% OK==========
\\& = \frac{1}{\gamma_N^{m+1}} V_{\psi}(x^*, x^1)
+ \frac{1}{2 \sigma} \sum_{k = 1}^{N} \frac{\|\nabla f(x^k)\|_*^2}{\gamma_k^{m-1}}. 
\end{aligned}    
\end{equation*}

Finally, by dividing by $\sum_{k= 1}^{N} \gamma_k^{-m}$,  we get the desired inequality
\begin{equation*}
f(\hat{x}) - f(x^*) \leq  \frac{1}{\sum_{k= 1}^{N} \gamma_k^{-m}} \left( \frac{V_{\psi}(x^*, x^1) }{\gamma_N^{m+1}} + \frac{1}{2 \sigma} \sum_{k= 1}^{N} \frac{\|\nabla f(x^k)\|_*^2}{\gamma_k^{m -1}} \right).
\end{equation*}

\end{proof}

As a special case of Theorem \ref{theo_main_ineq_mirror_desc}, {\color{black} with a special value of the parameter $m$,} we can deduce the well-known sub-optimal convergence rate  $O\left(\frac{\log(N)}{\sqrt{N}}\right)$ of Algorithm \ref{alg_mirror_descent}, see \cite{Beck2017book,Bubeck_book,Lan2020book,Nesterov_book}, with the following time-varying step size rule
\begin{equation}\label{steps_rules}
\gamma_k = \frac{\sqrt{2 \sigma}}{M_f \sqrt{k}} \quad  k = 1, 2, \ldots, N.
\end{equation}

\begin{corollary}\label{corollary_mirror_m_minus1}
Let $f$ be an $M_f$-Lipschitz convex function, and assume $ \|\nabla f(x)\|_* \leq M_f, \forall x \in Q$. Then for problem \eqref{main_constrained_prob}, by Algorithm \ref{alg_mirror_descent}, with  $m = -1$,  $V_{\psi}(x^*, x^1) \leq \theta$,  for some $\theta >0$, and with the time-varying step size given in \eqref{steps_rules}, 
it satisfies the following sub-optimal convergence rate
\begin{equation}\label{rate_mirror_k_minus1}
f(\tilde{x})  - f(x^*) \leq \frac{M_f \left(  \theta + 1 + \log(N)  \right)}{\sqrt{\sigma}} \frac{1}{\sqrt{N}} = O\left(\frac{\log(N)}{\sqrt{N}}\right), 
\end{equation}
where $\tilde{x} = \frac{1}{\sum_{k = 1}^{N} \gamma_k} \sum_{k= 1}^{N} \gamma_k x^k$. 
\end{corollary}
\begin{proof}
By setting $m = -1$ in \eqref{main_ineq_mirror_desc}, we get the following inequality
\begin{equation}\label{eq_3}
f(\tilde{x}) - f(x^*) \leq \frac{V_{\psi}(x^*, x^1) + \frac{1}{2 \sigma} \sum_{k = 1}^{N} \gamma_k^2 \|\nabla f(x^k)\|_*^2}{\sum_{k = 1}^{N} \gamma_k }, 
\end{equation}
where $\tilde{x} = \frac{1}{\sum_{k = 1}^{N} \gamma_k} \sum_{k= 1}^{N} \gamma_k x^k$. 

\noindent
When $\gamma_k = \frac{\sqrt{2 \sigma}}{M_f \sqrt{k}}, \, k = 1, 2, \ldots, N$, and since $\|\nabla f(x^k)\|_* \leq M_f$, $V_{\psi}(x^*, x^1) \leq \theta$, then by substitution in \eqref{eq_3} we find
$$
f(\tilde{x}) - f(x^*) \leq \frac{M_f}{\sqrt{2 \sigma}} \frac{\theta + \sum_{k = 1}^{N} \frac{1}{k} }{\sum_{k = 1}^{N} \frac{1}{\sqrt{k}}},
$$
where $\tilde{x} = \frac{1}{\sum_{k = 1}^{N} \frac{1}{\sqrt{k}}} \sum_{k= 1}^{N} \frac{1}{\sqrt{k}} x^k$. But
\begin{equation*}\label{eq_4}
\sum_{k = 1}^{N} \frac{1}{k} \leq 1 + \log(N), \quad \text{and} \quad  \sum_{k = 1}^{N} \frac{1}{\sqrt{k}} \geq 2 \sqrt{N+1} - 2, \quad \forall N \geq 1.
\end{equation*}
Therefore,
\begin{equation*}
    \begin{aligned}
      f(\tilde{x}) - f(x^*) &\leq  \frac{M_f}{\sqrt{2 \sigma}} \frac{\theta + 1  + \log(N)}{2 \sqrt{N+1} - 2} 
        \leq \frac{M_f}{\sqrt{ \sigma}} \frac{ 1 + \theta + \log(N)}{\sqrt{N}}, \quad \forall N \geq 1. 
    \end{aligned}
\end{equation*}
Where in the last inequality, we used the fact $2 \sqrt{2}( \sqrt{N+1} - 1) \geq \sqrt{N}, \, \forall N \geq 1$. 

%\bigskip 

%\noindent
%\textbf{Case 2 (adaptive rule).} When $\gamma_k = \frac{\sqrt{2 \sigma}}{\|\nabla f(x^k)\|_* \sqrt{k}}, \, k = 1, 2, \ldots, N$, and since $\|\nabla f(x^k)\|_* \leq M_f$, $V_{\psi}(x^*, x^1) \leq \theta$,  then by substitution in \eqref{eq_3} we find
%$$
%\begin{aligned}
%f(\tilde{x}) - f(x^*) & \leq  \frac{\theta + \sum_{k = 1}^{N} \frac{1}{k} }{\sum_{k = 1}^{N} \frac{\sqrt{2\sigma}}{M_f \sqrt{k}}  } 
%= \frac{M_f}{\sqrt{2 \sigma}} \frac{\theta + \sum_{k = 1}^{N} \frac{1}{k} }{\sum_{k = 1}^{N} \frac{1}{\sqrt{k}}} 
%\\& \leq \frac{M_f}{\sqrt{\sigma}} \frac{\theta + 1  + \log(N)}{\sqrt{N}}, 
%\end{aligned}
%$$
%where $\tilde{x} = \frac{1}{\sum_{k = 1}^{N} \left(\|\nabla f(x^k)\|_* \sqrt{k}\right)^{-1} } \sum_{k= 1}^{N} \left(\|\nabla f(x^k)\|_* \sqrt{k}\right)^{-1} x^k$. 
\end{proof}

Also, from Theorem \ref{theo_main_ineq_mirror_desc},  {\color{black} with a special value of the parameter $m$,} we can obtain the optimal convergence rate $O\left(\frac{1}{\sqrt{N}}\right)$ of Algorithm \ref{alg_mirror_descent}, with the time-varying step size given in \eqref{steps_rules}.

\begin{corollary}\label{corollary_mirror_m_0}
Let $f$ be an $M_f$-Lipschitz convex function, and assume $ \|\nabla f(x)\|_* \leq M_f, \forall x \in Q$. Then for problem \eqref{main_constrained_prob}, by Algorithm \ref{alg_mirror_descent}, with  $m=0$, $V_{\psi}(x^*, x^1) \leq \theta$,  for some $\theta >0$, and with the time-varying step size given in \eqref{steps_rules}, it satisfies the following optimal convergence rate
\begin{equation}\label{rate_mirror_k_0}
f(\overline{x})  - f(x^*) \leq \frac{M_f \left( 2+ \theta \right)}{\sqrt{2\sigma}} \frac{1}{\sqrt{N}} = O\left(\frac{1}{\sqrt{N}}\right), 
\end{equation}
where $\overline{x} = \frac{1}{N} \sum_{k= 1}^{N} x^k$. 
\end{corollary}
\begin{proof}
By setting $m = 0$ in \eqref{main_ineq_mirror_desc}, we get the following inequality
\begin{equation}\label{eq_5}
f(\overline{x}) - f(x^*) \leq \frac{1}{N} \left( \frac{V_{\psi}(x^*, x^1)}{\gamma_N} + \frac{1}{2 \sigma} \sum_{k = 1}^{N} \gamma_k \|\nabla f(x^k)\|_*^2  \right), 
\end{equation}
where $\overline{x}= \frac{1}{N} \sum_{k= 1}^{N} x^k$. 

\noindent
When $\gamma_k = \frac{\sqrt{2 \sigma}}{M_f \sqrt{k}}, \, k = 1, 2, \ldots, N$, and since $\|\nabla f(x^k)\|_* \leq M_f, V_{\psi} (x^*, x^1) \leq \theta$,  then by substitution in \eqref{eq_5} we find
$$
f(\overline{x}) - f(x^*) \leq  \frac{1}{N} \left(  \frac{\theta M_f \sqrt{N} }{\sqrt{2 \sigma}} + \frac{M_f}{\sqrt{2 \sigma}} \sum_{k = 1}^{N} \frac{1}{\sqrt{k}} \right).
$$
But
\begin{equation*}\label{eq_6}
\sum_{k = 1}^{N} \frac{1}{\sqrt{k}} \leq 2 \sqrt{N}, \quad \forall N \geq 1.
\end{equation*}
Therefore,
\begin{equation*}
f(\overline{x}) - f(x^*) \leq  \frac{1}{N} \frac{M_f}{\sqrt{2 \sigma}} \left( \theta \sqrt{N}  + 2 \sqrt{N}\right) = \frac{M_f (2 + \theta)}{\sqrt{2 \sigma}}  \cdot \frac{1}{\sqrt{N}}.
\end{equation*}

%\bigskip 

%\noindent
%\textbf{Case 2 (adaptive rule).} When $\gamma_k = \frac{\sqrt{2 \sigma}}{\|\nabla f(x^k)\|_* \sqrt{k}}, \, k = 1, 2, \ldots, N$, and since $\|\nabla f(x^k)\|_* \leq M_f, V_{\psi}(x^*, x^1) \leq \theta$, then by substitution in \eqref{eq_5} we find
%$$
%\begin{aligned}
%f(\overline{x}) - f(x^*) &\leq  \frac{1}{N} \left( \frac{ \theta \sqrt{N} \|\nabla f(x^N)\|_*}{\sqrt{2 \sigma}} + \frac{1}{2 \sigma} \sum_{k = 1}^{N} \frac{\sqrt{2 \sigma}}{\sqrt{k}}  \|\nabla f(x^k)\|_*\right)
%\\& \leq  \frac{1}{N} \left( \frac{ \theta M_f \sqrt{N} }{\sqrt{2 \sigma}} + \frac{M_f}{\sqrt{2 \sigma}} \sum_{k = 1}^{N} \frac{1}{\sqrt{k}}\right)
%\leq \frac{1}{N} \frac{M_f}{\sqrt{2\sigma}} \left( \sqrt{N}\theta + 2  \sqrt{N}  \right) 
%\\& = \frac{M_f (2 + \theta)}{\sqrt{2 \sigma}} \cdot \frac{1}{\sqrt{N}}.
%\end{aligned}
%$$
\end{proof}

Now, from Theorem \ref{theo_main_ineq_mirror_desc},   with any fixed $m \geq 1$, we can obtain the optimal convergence rate of Algorithm \ref{alg_mirror_descent} $O\left(\frac{1}{\sqrt{N}}\right)$, with the time-varying step size given in \eqref{steps_rules}.

\begin{corollary}\label{corollary_mirror_m_all}
Let $f$ be an $M_f$-Lipschitz convex function, and assume $ \|\nabla f(x)\|_* \leq M_f, \forall x \in Q$. Then for problem \eqref{main_constrained_prob}, by Algorithm \ref{alg_mirror_descent}, with any $m \geq 1$, $V_{\psi}(x^*, x^1) \leq \theta$, for some $\theta >0$, and with the time-varying step size given in \eqref{steps_rules}, it satisfies the following optimal convergence rate
\begin{equation}\label{rate_mirror_m_all}
f(\hat{x})  - f(x^*) \leq \frac{M_f (m + 2) (1+ \theta)}{2 \sqrt{2 \sigma}} \cdot \frac{1}{\sqrt{N}}, 
\end{equation}
where $\hat{x} = \frac{1}{\sum_{k= 1}^{N} \gamma_k^{-m}} \sum_{k = 1}^{N} \gamma_k^{-m} x^k $. 
\end{corollary}
\begin{proof} 
When $\gamma_k = \frac{\sqrt{2 \sigma}}{M_f \sqrt{k}}, \, k = 1, 2, \ldots, N$, and since $\|\nabla f(x^k)\|_* \leq M_f$, $V_{\psi}(x^*, x^1) \leq \theta$,  then by substitution in \eqref{main_ineq_mirror_desc} we find
$$
f(\hat{x}) - f(x^*) \leq \frac{M_f}{\sqrt{2 \sigma}} \cdot \frac{1}{\sum_{k = 1}^{N} \left(\sqrt{k}\right)^{m} }  \left(  \theta \left(\sqrt{N}\right)^{m+1}  +  \sum_{k = 1}^{N} \left(\sqrt{k}\right)^{m-1} 
 \right).
$$

But, for any $m \geq 1$ and $N \geq 1$, 
\begin{equation*}
\int_{0}^{N}\left(\sqrt{k}\right)^{m} dk \leq \sum_{k = 1}^{N}  \left(\sqrt{k}\right)^{m} \Longrightarrow \sum_{k = 1}^{N}  \left(\sqrt{k}\right)^{m} \geq \frac{2\left(\sqrt{N}\right)^{m+2}}{m+2},
\end{equation*}

and
$$
\sum_{k = 1}^{N} \left(\sqrt{k}\right)^{m-1} \leq N \left(\sqrt{N}\right)^{m-1} = \left(\sqrt{N}\right)^{m+1}, \quad \forall m \geq 1, N\geq 1  . 
$$

Therefore,
\begin{equation*}
\begin{aligned}
    f(\hat{x}) - f(x^*) & \leq  \frac{M_f}{\sqrt{2 \sigma}} \cdot  \frac{m + 2}{2 \left(\sqrt{N}\right)^{m+2}}  \left(  \theta \left(\sqrt{N}\right)^{m+1} + \left(\sqrt{N}\right)^{m+1}\right) 
    \\& = \frac{M_f (m + 2) (1+ \theta)}{2 \sqrt{2 \sigma}} \cdot \frac{1}{\sqrt{N}} = O \left(\frac{1}{\sqrt{N}}\right). 
\end{aligned}
\end{equation*}

\end{proof}

\begin{remark}\label{remark_k_gretear_minus1_mirror}
%By setting any $m> -1$ in \eqref{main_ineq_mirror_desc} (in Theorem \ref{theo_main_ineq_mirror_desc}), the convergence rate of Algorithm \ref{alg_mirror_descent} with the time-varying step sizes given in \eqref{steps_rules}, will be optimal $O\left(\frac{1}{\sqrt{N}}\right)$ without presence of a $\log(N)$ factor.  
In comparison with the sub-optimal convergence rate \eqref{rate_mirror_k_minus1}, when $m \geq 1$, the weighting scheme $\frac{1}{\sum_{k= 1}^{N} \gamma_k^{-m}} \sum_{k = 1}^{N} \gamma_k^{-m} x^k$ assigns smaller weights to the initial points and larger weights to the most recent points that generated by Algorithm \ref{alg_mirror_descent}. This fact will be shown in numerical experiments. 
\end{remark}

%%%%%%%%%%%%%%%%%%%
\section{Mirror descent method for constrained composite problem}\label{sectinMirrorC}

In this section, we will consider a more general problem than problem \eqref{main_constrained_prob}, namely
\begin{equation}\label{main_constrained_prob_composition}
\min_{x \in Q} \left\{F(x) = f(x ) + h(x)\right\},
\end{equation}
where $Q = \operatorname{dom} (h) \cap \operatorname{dom}(\partial \psi), f$ and $h$ satisfy the following assumption
\begin{assumption}\label{assump_composite_problem}(properties of $f$ and $h$)
\begin{enumerate}
\item $f$ and $h$ are proper closed and convex;
\item $ \operatorname{dom}(h)   \subseteq  \operatorname{int} (\operatorname{dom}(f))$, where $\operatorname{int} (A)$ denotes the interior of the set $A$; 
\item  $f$ is $M_f$-Lipschitz continuous on $\operatorname{dom}(h)$ for any $M_f > 0 $.
\end{enumerate}
\end{assumption}
We also, assume that the optimal set of the problem \eqref{main_constrained_prob_composition} is nonempty.% and denoted by $X^*$. 

For the problem \eqref{main_constrained_prob_composition}, with an additional assumption that $h$ is non-negative, we mention the following well-defined algorithm \cite{Beck2017book}. 

\begin{algorithm}[H]
\caption{Mirror-C Descent Method.}\label{alg_mirrorC_descent}
\begin{algorithmic}[1]
\REQUIRE step sizes $\{\gamma_k\}_{k \geq 1}$,  initial point $x^1 \in Q$ {\color{black} s.t. \eqref{cond_x1xstar} holds}, number of iterations $N$.
\FOR{$k= 1, 2, \ldots, N$}
\STATE Calculate $\nabla f(x^k) \in \partial f(x^k)$,
\STATE $x^{k+1} = \arg\min_{x \in Q} \left\{ \gamma_k \langle x, \nabla f(x^k)  \rangle + \gamma_k h(x) +  V_{\psi}(x,x^k) \right\} $.
\ENDFOR
%\ENSURE $\hat{x} = \frac{1}{k} \sum_{i = 1}^{k} x^i$.
\end{algorithmic}
\end{algorithm}

Let us mention the following key lemma for the analysis of the Algorithm \ref{alg_mirrorC_descent}. 

\begin{lemma}\cite{Beck2017book}\label{lemma_mirrorC}
Let %(see Theorem 9.12 page 253 from \cite{Beck2017book})
\begin{enumerate}
\item $\psi: \mathbf{E} \longrightarrow (-\infty, \infty]$ be a proper closed and convex function differentiable over $\operatorname{dom}(\partial \psi)$,

\item $\varphi: \mathbf{E} \longrightarrow (-\infty, \infty]$ be a proper closed and convex function satisfying  $\operatorname{dom}(\partial \varphi) \subseteq \operatorname{dom}(\partial \psi)$,

\item $\psi + \mathbb{I}_{\operatorname{dom}(\varphi)}$ be a $\sigma$-strongly convex, where $\mathbb{I}_{A}$ is the indicator function of the set $A$. 
\end{enumerate}
Assume that $b \in \operatorname{dom}(\partial \psi)$, and let $a$ be defined by
$$
a = \arg\min_{x \in \mathbf{E}} \left\{ \varphi(x) + V_{\psi}(x, b) \right\}.
$$
Then %$a \in \operatorname{dom}(\partial(\psi))$, 
we have
\begin{equation}\label{eq_main_lemma}
\left\langle \nabla \psi(b) - \nabla \psi(a), u - a \right\rangle \leq \varphi(u) - \varphi(a),	\quad \forall u \in \operatorname{dom}(\varphi).
\end{equation}
\end{lemma}

\begin{theorem}\label{theo_main_ineq_mirrorC}
Suppose that Assumption \ref{assump_composite_problem} holds and $h$ is a non-negative function. Then for problem \eqref{main_constrained_prob_composition}, by Algorithm \ref{alg_mirrorC_descent}, with a positive non-increasing sequence of step sizes $\{\gamma_k\}_{k \geq 1}$, for any fixed $ -1 \leq m \leq 0$, it satisfies the following inequality
\begin{equation}\label{main_ineq_mirrorC}
F(\hat{x}) - F(x^*) \leq  \frac{1}{\sum_{k= 1}^{N} \gamma_k^{-m}} \left( \frac{h(x^1)}{\gamma_1^m} +  \frac{V_{\psi}(x^*, x^1) }{\gamma_N^{m+1}} + \frac{1}{2 \sigma} \sum_{k= 1}^{N} \frac{\|\nabla f(x^k)\|_*^2}{\gamma_k^{m -1}} \right),
\end{equation}
where $\hat{x} = \frac{1}{\sum_{k= 1}^{N} \gamma_k^{-m}} \sum_{k = 1}^{N} \gamma_k^{-m} x^k $,  and $x^*$ is an optimal solution of \eqref{main_constrained_prob_composition}.
\end{theorem}

\begin{proof}
From \eqref{eq_main_lemma}, with $b = x^k, a = x^{k+1} $ and  $\varphi(x) = \gamma_k \langle x,  \nabla f(x^k) \rangle + \gamma_k h(x)$, we get the following inequality
\begin{equation}\label{eq_s1}
\langle \nabla \psi(x^k) -  \nabla \psi(x^{k+1}), u - x^{k+1} \rangle \leq \gamma_k  \langle u - x^{k+1}, \nabla f(x^k) \rangle + \gamma_k h(u) -  \gamma_k h(x^{k+1}). 
\end{equation}
Now, from three points identity \eqref{eq_three_points}, with $a = x^{k+1}, b = x^k$ and $c = u$, we get
\begin{equation}\label{eq_s2}
\langle \nabla \psi(x^k) -  \nabla \psi(x^{k+1}), u - x^{k+1} \rangle = V_{\psi} (u, x^{k+1}) + V_{\psi} (x^{k+1}, x^k) - V_{\psi} (u, x^{k}). 
\end{equation}
By combining \eqref{eq_s1} and  \eqref{eq_s2}, we get the following inequality
\begin{equation}\label{eq_hfghfkg}
\begin{aligned}
    V_{\psi} (u, x^{k+1}) +  V_{\psi} (x^{k+1}, x^k) -  V_{\psi} (u, x^k) & \leq \gamma_k \langle u - x^{k+1}, \nabla f(x^k) \rangle 
    \\& \;\;\;\; + \gamma_k h(u) -  \gamma_k h(x^{k+1}).
\end{aligned}
\end{equation}
From \eqref{eq_hfghfkg},  we find
\begin{equation}\label{eq_dhfhsc}
\begin{aligned}
\gamma_k \langle \nabla f(x^k), x^{k} - u  \rangle  & +   \gamma_k h(x^{k+1})    - \gamma_k h(u) 
 \leq V_{\psi} (u, x^k) - V_{\psi} (u, x^{k+1})
\\& \qquad \quad - V_{\psi} (x^{k+1}, x^k)  + \langle \gamma_k \nabla f(x^k) , x^k - x^{k+1} \rangle.
\end{aligned}
\end{equation}
Therefore, by \eqref{eq_breg} and Fenchel-Young inequality \eqref{Fenchel_Young_ineq} with $\lambda = \sigma$, from \eqref{eq_dhfhsc} we get the following
\begin{equation}\label{eq_s3}
\begin{aligned}
\gamma_k \langle \nabla f(x^k), x^{k} - u  \rangle  & +   \gamma_k h(x^{k+1})    - \gamma_k h(u) 
\leq V_{\psi} (u, x^k) - V_{\psi} (u, x^{k+1})
\\& - \frac{\sigma}{2} \| x^{k+1} - x^k\|^2  + \frac{\gamma_k^2}{2 \sigma} \|\nabla f(x^k)\|_*^2
 + \frac{\sigma}{2} \| x^{k+1} - x^k\|^2.
 \\& = V_{\psi} (u, x^{k}) - V_{\psi} (u, x^{k+1})  + \frac{\gamma_k^2}{2 \sigma} \|\nabla f(x^k)\|_*^2. 
\end{aligned}
\end{equation}
Since $f$ is a convex function, it holds
$$
f(u) - f(x^k) \geq \langle \nabla f(x^k), u - x^k \rangle \Longrightarrow f(x^k) - f(u)  \leq \langle \nabla f(x^k), x^k  - u\rangle, 
$$
and by setting $u = x^*$ in \eqref{eq_s3}, we get the following inequality
$$
\begin{aligned}
    \gamma_k \left(f(x^k) - f( x^*) \right)+    \gamma_k h(x^{k+1})    - \gamma_k h(x^*)  & \leq V_{\psi} (x^*, x^k) - V_{\psi} (x^*, x^{k+1}) 
    \\& \;\;\;\; + \frac{\gamma_k^2}{2 \sigma} \|\nabla f(x^k)\|_*^2,
\end{aligned}
$$
which means 
\begin{equation}\label{eq_s8}
f(x^k)  + h(x^{k+1}) - F(x^*) \leq \frac{1}{\gamma_k} \left(V_{\psi} (x^*, x^k) - V_{\psi} (x^*, x^{k+1}) \right) + \frac{\gamma_k}{2 \sigma} \|\nabla f(x^k)\|_*^2 .
\end{equation}
By multiplying both sides of \eqref{eq_s8} by $\frac{1}{\gamma_k^m}$, and taking the sum form $1$ to $N$, we get
\begin{equation*}
\begin{aligned}
    \sum_{k = 1 }^{N} \frac{1}{\gamma_k^m}  \left(f(x^k)  + h(x^{k+1}) - F(x^*) \right) & \leq \sum_{k = 1 }^{N} \frac{1}{\gamma_k^{m+1}} \left(V_{\psi} (x^*, x^k) - V_{\psi} (x^*, x^{k+1}) \right) 
    \\& \;\;\;\; + \frac{1}{2 \sigma} \sum_{k = 1 }^{N} \frac{\|\nabla f(x^k)\|_*^2}{\gamma_k^{m-1}}  .
\end{aligned}
\end{equation*}
Now, in a similar way as in the proof of Theorem \ref{theo_main_ineq_mirror_desc}, we conclude 
\begin{equation}\label{eq_s9}
\sum_{k = 1 }^{N} \frac{1}{\gamma_k^m}  \left(f(x^k)  + h(x^{k+1}) - F(x^*) \right) \leq   \frac{V_{\psi}(x^*, x^1)}{\gamma_N^{m+1}} + \frac{1}{2 \sigma} \sum_{k = 1}^{N} \frac{\|\nabla f(x^k)\|_*^2}{\gamma_k^{m-1}}.
\end{equation}
By adding the term $\frac{h(x^1)}{\gamma_1^m} - \frac{h(x^{N+1})}{\gamma_N^m}$ to both sides of \eqref{eq_s9}, we  get 
\\ the following
\begin{equation*}
    \begin{aligned}
      & \sum_{k = 2}^{N} \frac{1}{\gamma_k^m}  \left(f(x^k)  + h(x^{k+1}) - F(x^*) \right) + \frac{f(x^1)}{\gamma_1^m} +   \frac{h(x^2)}{\gamma_1^m} - \frac{F(x^*)}{\gamma_1^m} +\frac{h(x^1)}{\gamma_1^m} 
      \\& \;\;\;\; - \frac{h(x^{N+1})}{\gamma_N^m}  \leq
      \frac{h(x^1)}{\gamma_1^m} - \frac{h(x^{N+1})}{\gamma_N^m} + \frac{V_{\psi}(x^*, x^1)}{\gamma_N^{m+1}} 
    + \frac{1}{2 \sigma} \sum_{k = 1}^{N} \frac{\|\nabla f(x^k)\|_*^2}{\gamma_k^{m-1}},
    \end{aligned}
\end{equation*}
Because of $h$ is a non-negative, we have $\frac{h(x^{N+1})}{\gamma_N^m} > 0$, and then we get the following
$$
\begin{aligned}
      & \sum_{k = 2}^{N}  \left( \frac{1}{\gamma_k^m} f(x^k)  +  \frac{1}{\gamma_k^m}  h(x^{k+1}) - \frac{1}{\gamma_k^m}  F(x^*) \right) + \frac{1}{\gamma_1^m} \left(F(x^1) - F(x^*)\right) +   \frac{h(x^2)}{\gamma_1^m}    
      \\& \;\;\;\; - \frac{h(x^{N+1})}{\gamma_N^m} \leq 
      \frac{h(x^1)}{\gamma_1^m} + \frac{V_{\psi}(x^*, x^1)}{\gamma_N^{m+1}}  + \frac{1}{2 \sigma} \sum_{k = 1}^{N} \frac{\|\nabla f(x^k)\|_*^2}{\gamma_k^{m-1}}, 
    \end{aligned}
$$
which means
$$
\begin{aligned}
      & \sum_{k = 2}^{N}  \left( \frac{1}{\gamma_k^m} f(x^k)  - \frac{1}{\gamma_k^m}  F(x^*) \right) + \sum_{k = 2}^{N}  \frac{1}{\gamma_k^m} h(x^{k+1}) + 
      \frac{1}{\gamma_1^m} \left(F(x^1) - F(x^*)\right) 
      \\& \;\;\;\; +   \frac{h(x^2)}{\gamma_1^m}  - \frac{h(x^{N+1})}{\gamma_N^m}  \leq
      \frac{h(x^1)}{\gamma_1^m} + \frac{V_{\psi}(x^*, x^1)}{\gamma_N^{m+1}}  + \frac{1}{2 \sigma} \sum_{k = 1}^{N} \frac{\|\nabla f(x^k)\|_*^2}{\gamma_k^{m-1}}. 
    \end{aligned}
$$
Thus, 
$$
\begin{aligned}
      & \sum_{k = 2}^{N}  \left( \frac{1}{\gamma_k^m} f(x^k)  - \frac{1}{\gamma_k^m}  F(x^*) \right) + \sum_{k = 3}^{N + 1}  \frac{1}{\gamma_{k-1}^m} h(x^{k}) + 
      \frac{1}{\gamma_1^m} \left(F(x^1) - F(x^*)\right) 
      \\& \;\;\;\; +   \frac{h(x^2)}{\gamma_1^m}  - \frac{h(x^{N+1})}{\gamma_N^m}  \leq
      \frac{h(x^1)}{\gamma_1^m} + \frac{V_{\psi}(x^*, x^1)}{\gamma_N^{m+1}}  + \frac{1}{2 \sigma} \sum_{k = 1}^{N} \frac{\|\nabla f(x^k)\|_*^2}{\gamma_k^{m-1}}. 
    \end{aligned}
$$
Therefore, 
$$
\begin{aligned}
      & \sum_{k = 2}^{N}  \left( \frac{1}{\gamma_k^m} f(x^k)  - \frac{1}{\gamma_k^m}  F(x^*) \right) +  
      \frac{1}{\gamma_1^m} \left(F(x^1) - F(x^*)\right) +  \sum_{k = 2}^{N}  \frac{1}{\gamma_{k-1}^m} h(x^{k}) 
       \leq
      \\& \;\;\;\; \frac{h(x^1)}{\gamma_1^m} + \frac{V_{\psi}(x^*, x^1)}{\gamma_N^{m+1}} 
       + \frac{1}{2 \sigma} \sum_{k = 1}^{N} \frac{\|\nabla f(x^k)\|_*^2}{\gamma_k^{m-1}}. 
\end{aligned}
$$
Thus, we get
$$
\begin{aligned}
      & \sum_{k = 2}^{N}  \left( \frac{1}{\gamma_k^m} f(x^k) +\frac{1}{\gamma_{k-1}^m} h(x^k)  - \frac{1}{\gamma_k^m}  F(x^*) \right) +  
      \frac{1}{\gamma_1^m} \left(F(x^1) - F(x^*)\right) 
       \leq
      \\& \;\;\;\; \frac{h(x^1)}{\gamma_1^m} + \frac{V_{\psi}(x^*, x^1)}{\gamma_N^{m+1}} 
       + \frac{1}{2 \sigma} \sum_{k = 1}^{N} \frac{\|\nabla f(x^k)\|_*^2}{\gamma_k^{m-1}}. 
    \end{aligned}
$$
Because of the sequence $\{\gamma_k\}_{k \geq 1}$ non-increasing and $-1 \leq m \leq 0$, we have $\frac{1}{\gamma_{k-1}^m} \geq \frac{1}{\gamma_{k}^m} $, and then we conclude
$$
\begin{aligned}
      & \sum_{k = 2}^{N}  \left( \frac{1}{\gamma_k^m} f(x^k) +\frac{1}{\gamma_{k}^m} h(x^k)  - \frac{1}{\gamma_k^m}  F(x^*) \right) +  
      \frac{1}{\gamma_1^m} \left(F(x^1) - F(x^*)\right) 
       \leq
     \\& \;\;\;\; \frac{h(x^1)}{\gamma_1^m} + \frac{V_{\psi}(x^*, x^1)}{\gamma_N^{m+1}} 
      + \frac{1}{2 \sigma} \sum_{k = 1}^{N} \frac{\|\nabla f(x^k)\|_*^2}{\gamma_k^{m-1}},
    \end{aligned}
$$
which means
$$
\sum_{k = 1}^{N}\frac{1}{\gamma_k^m}  \left(  F(x^k) -  F(x^*) \right) \leq 
\frac{h(x^1)}{\gamma_1^m} + \frac{V_{\psi}(x^*, x^1)}{\gamma_N^{m+1}} + \frac{1}{2 \sigma} \sum_{k = 1}^{N} \frac{\|\nabla f(x^k)\|_*^2}{\gamma_k^{m-1}}. 
$$
From the convexity of the function $F$, we have
$$
\begin{aligned}
    & \left(\sum_{k = 1}^{N}\frac{1}{\gamma_k^m}\right) \left[ F\left(\frac{1}{\sum_{k= 1}^{N} \gamma_k^{-m}} \sum_{k = 1}^{N} \gamma_k^{-m} x^k\right) - F(x^*) \right] \leq  
    \\&\;\;\;\;  \leq \sum_{k = 1}^{N}\frac{1}{\gamma_k^m}  \left(  F(x^k) -  F(x^*) \right). 
\end{aligned} 
$$
Therefore, by setting $\hat{x} = \frac{1}{\sum_{k= 1}^{N} \gamma_k^{-m}} \sum_{k = 1}^{N} \gamma_k^{-m} x^k $, we get the following inequality
\begin{equation}\label{ddd}
\left(\sum_{k = 1}^{N}\frac{1}{\gamma_k^m}\right) \left( F(\hat{x}) - F(x^*) \right)  \leq \frac{h(x^1)}{\gamma_1^m} + \frac{V_{\psi}(x^*, x^1)}{\gamma_N^{m+1}} + \frac{1}{2 \sigma} \sum_{k = 1}^{N} \frac{\|\nabla f(x^k)\|_*^2}{\gamma_k^{m-1}}.
\end{equation}
By dividing both sides of \eqref{ddd} by $\sum_{k = 1}^{N}\gamma_k^{-m}$, we get the desired inequality \eqref{main_ineq_mirrorC}. 

\end{proof}

As a special case of Theorem \ref{theo_main_ineq_mirrorC}, {\color{black} with a special value of the parameter $m$,}  we can deduce the well-known sub-optimal convergence rate of the mirror-C descent method (Algorithm \ref{alg_mirrorC_descent}), namely $O\left(\frac{\log(N)}{\sqrt{N}}\right)$, with the time-varying step size given in \eqref{steps_rules} (see \cite{Beck2017book,Bubeck_book,Lan2020book,Nesterov_book}).
%\begin{equation}\label{steps_rules}
%\gamma_k = \frac{\sqrt{2 \sigma}}{M_f \sqrt{k}}, \quad \text{or} \quad \gamma_k = \frac{\sqrt{2 \sigma}}{\|\nabla f(x^k)\|_* \sqrt{k}}, \quad  k = 1, 2, \ldots, N,
%\end{equation}.

\begin{corollary}\label{corollary_mirror_m_minus1_C}
Suppose that Assumption \ref{assump_composite_problem} holds and $h$ is a non-negative function. Then for problem \eqref{main_constrained_prob_composition}, by Algorithm \ref{alg_mirrorC_descent}, with  $m=-1$, $V_{\psi}(x^*, x^1) \leq \theta$,  for some $\theta >0$, and with the time-varying step size given in \eqref{steps_rules}, 
it satisfies the following sub-optimal convergence rate
\begin{equation}\label{rate_mirror_k_minus1_C}
\begin{aligned}
    F(\tilde{x})  - F(x^*) & \leq  \frac{M_f}{\sqrt{\sigma}} \left( \frac{\sqrt{2 \sigma}}{M_f} h (x^1) + \theta + 1 + \log(N) \right)  \frac{1}{\sqrt{N}} 
    \\& = O\left(\frac{\log(N)}{\sqrt{N}}\right) ,
\end{aligned} 
\end{equation}
where $\tilde{x} = \frac{1}{\sum_{k = 1}^{N} \gamma_k} \sum_{k= 1}^{N}\gamma_k x^k$. 
\end{corollary}
\begin{proof}
By setting $m = -1$ in \eqref{main_ineq_mirrorC}, we get the following inequality
\begin{equation}\label{eqqq}
F(\tilde{x}) - F(x^*) \leq \frac{ \gamma_1 h(x^1) +  V_{\psi}(x^*, x^1) + \frac{1}{2 \sigma} \sum_{k = 1}^{N} \gamma_k^2 \|\nabla f(x^k)\|_*^2}{\sum_{k = 1}^{N} \gamma_k }, 
\end{equation}
where $\tilde{x} = \frac{1}{\sum_{k = 1}^{N} \gamma_k} \sum_{k= 1}^{N}\gamma_k x^k $. 

\noindent
When $\gamma_k = \frac{\sqrt{2 \sigma}}{M_f \sqrt{k}}, \, k = 1, 2, \ldots, N$, and since $\|\nabla f(x^k)\|_* \leq M_f, V_{\psi}(x^*, x^1) \leq \theta$, then by substitution in \eqref{eqqq}, with $\tilde{x} = \frac{1}{\sum_{k = 1}^{N} \frac{1}{\sqrt{k}}} \sum_{k= 1}^{N} \frac{1}{\sqrt{k}} x^k$,  we find
\begin{equation*}
    \begin{aligned}
    F(\tilde{x}) - F(x^*) & \leq \frac{M_f}{\sqrt{2 \sigma}} \frac{ \frac{\sqrt{2 \sigma}}{M_f} h(x^1) +  \theta + \sum_{k = 1}^{N} \frac{1}{k} }{\sum_{k = 1}^{N} \frac{1}{\sqrt{k}}}
    \\& \leq  \frac{M_f}{\sqrt{2 \sigma}} \frac{ \frac{\sqrt{2 \sigma}}{M_f} h(x^1) +1+ \theta + \log(N)}{2 \sqrt{N+1} - 2} 
      \\&  \leq \frac{M_f}{\sqrt{ \sigma}} \frac{ \frac{\sqrt{2 \sigma}}{M_f} h(x^1) + 1+ \theta  + \log(N)}{\sqrt{N}}. 
    \end{aligned}
\end{equation*}
Where in the last inequality, we used the fact $2 \sqrt{2}( \sqrt{N+1} - 1) \geq \sqrt{N}, \, \forall N \geq 1$.

%\textbf{Case 2 (adaptive rule).} When $\gamma_k = \frac{\sqrt{2 \sigma}}{\|\nabla f(x^k)\|_* \sqrt{k}}, \, k = 1, 2, \ldots, N$, and since $\|\nabla f(x^k)\|_* \leq M_f, V_{\psi}(x^*, x^1) \leq \theta$, then by substitution in \eqref{eqqq}, we find
%$$
%\begin{aligned}
%F(\tilde{x}) - F(x^*) &\leq  \frac{ \frac{\sqrt{2 \sigma}}{M_f} h(x^1) +  \theta + \sum_{k = 1}^{N} \frac{1}{k} }{\sum_{k = 1}^{N} \frac{\sqrt{2\sigma}}{M_f \sqrt{k}}  } 
%\\& \leq \frac{M_f}{\sqrt{\sigma}} \frac{ \frac{\sqrt{2 \sigma}}{M_f} h(x^1) +  1 + \theta + \log(N)}{\sqrt{N}},
%\end{aligned}
%$$
%where $\tilde{x} = \frac{1}{\sum_{k = 1}^{N} \left(\|\nabla f(x^k)\|_* \sqrt{k}\right)^{-1} } \sum_{k= 1}^{N} \left(\|\nabla f(x^k)\|_* \sqrt{k}\right)^{-1} x^k$.
\end{proof}

Also, from Theorem \ref{theo_main_ineq_mirrorC}, {\color{black} with a special value of the parameter $m$,}  we can obtain the optimal convergence rate $O\left(\frac{1}{\sqrt{N}}\right)$ of Algorithm \ref{alg_mirrorC_descent}, with the time-varying step size given in \eqref{steps_rules}.

\begin{corollary}\label{corollary_mirror_m_0_C}
Suppose that Assumption \ref{assump_composite_problem} hold and $h$ is a non-negative function, with $h(x^1) = 0$. Then for problem \eqref{main_constrained_prob_composition}, by Algorithm \ref{alg_mirrorC_descent}, with $m = 0$, $V_{\psi}(x^*, x^1) \leq \theta$,  for some $\theta >0$, and with the time-varying step size given in \eqref{steps_rules}, it satisfies the following optimal convergence rate
\begin{equation*}\label{rate_mirror_k_0_C}
F(\overline{x})  - F(x^*) \leq  \frac{ (2 + \theta)M_f }{\sqrt{2\sigma}} \cdot \frac{1}{\sqrt{N}} = O\left(\frac{1}{\sqrt{N}}\right), 
\end{equation*}
where $\overline{x} = \frac{1}{N} \sum_{k= 1}^{N} x^k$.
\end{corollary}
\begin{proof}
By setting $m = 0$ in \eqref{main_ineq_mirrorC}, and since $h(x^1) = 0$,  we get the following inequality
\begin{equation}\label{eee}
F(\overline{x}) - F(x^*) \leq \frac{1}{N} \left( \frac{V_{\psi}(x^*, x^1)}{\gamma_N} + \frac{1}{2 \sigma} \sum_{k = 1}^{N} \gamma_k \|\nabla f(x^k)\|_*^2 \right), 
\end{equation}
where $\overline{x} = \frac{1}{N} \sum_{k= 1}^{N}  x^k$. 

\noindent
When $\gamma_k = \frac{\sqrt{2 \sigma}}{M_f \sqrt{k}}, \, k = 1, 2, \ldots, N$, and since $\|\nabla f(x^k)\|_* \leq M_f, V_{\psi}(x^*, x^1) \leq \theta$, then by substitution in \eqref{eee} we find
\begin{equation*}
    \begin{aligned}
    F(\overline{x}) - F(x^*) & \leq  \frac{1}{N} \left( \frac{\theta M_f \sqrt{N} }{\sqrt{2 \sigma}} + \frac{M_f}{\sqrt{2 \sigma}} \sum_{k = 1}^{N} \frac{1}{\sqrt{k}} \right)
    \leq \frac{1}{N} \frac{M_f}{\sqrt{2 \sigma}} \left(\theta \sqrt{N} + 2 \sqrt{N}\right)
   \\& = \frac{M_f(2 + \theta)}{\sqrt{2 \sigma}} \cdot \frac{1}{\sqrt{N}}.
    \end{aligned}
\end{equation*}
\end{proof}

%\begin{remark}\label{remark_k_gretear_minus1_mirror_C}

%{\color{black}In a similar way as in Corollary \ref{corollary_mirror_m_all}, for any $m \geq  1$ in \eqref{main_ineq_mirrorC} with the assumption that $h(x^1) = 0$ we get the same result \eqref{rate_mirror_m_all}, and the convergence rate of Algorithm \ref{alg_mirrorC_descent} with the time-varying step sizes given in \eqref{steps_rules}, will be optimal}.   In comparison with the sub-optimal convergence rate \eqref{rate_mirror_k_minus1_C}, when $m \geq 1$, the weighting scheme $\frac{1}{\sum_{k= 1}^{N} \gamma_k^{-m}} \sum_{k = 1}^{N} \gamma_k^{-m} x^k$ assigns smaller weights to the initial points and larger weights to the most recent points that generated by Algorithm \ref{alg_mirrorC_descent}. %For this fact see the next section which is devoted to the numerical experiments.
%\end{remark}

{\color{black}
\section{Mirror descent method for problems with functional constraints}\label{section_functio_const}

In this section, we focus on a more general class of problems, which is a class of optimization problems with functional constraints (inequality constraints)
\begin{equation}\label{general_problem_1}
\min\left\{f(x): \;\; x \in Q \;\;\text{and} \;\; g_i(x)\leq 0 \;\; \text{for all} \;\; i = 1 , 2, \ldots, p \right\},
\end{equation}
where $f$ and $g_i$ (for all $i =1, \ldots, p$) are convex and non-smooth functions given on a compact convex set $Q$, this means that they are Lipschitz functions, i.e., there is $M_f > 0$ and  $M_{g_i} > 0  \; (i = 1, 2, \ldots, p)$, such that $\|\nabla f(x)\|_* \leq M_f$ and $\|\nabla g_i(x)\|_* \leq M_{g_i} (i = 1, 2, \ldots, p)$. 

It is clear that instead of a set of functionals  $\{g_i(\cdot)\}_{i=1}^{p}$ we can see one functional constraint $g: Q \longrightarrow \mathbb{R}$, such that $g(x) = \max_{1 \leq i \leq p}\{g_i(x)\}$. Therefore, by this setting, the problem \eqref{general_problem_1} will be equivalent to the following constrained minimization problem  
\begin{equation}\label{general_problem}
 \min\limits_{x\in Q,\, g(x) \leq 0} f(x).
\end{equation} 

Let $x^*$ be a solution of \eqref{general_problem}, we say that $\hat{x} \in Q$ is an $\varepsilon$-solution to \eqref{general_problem} if $f(\hat{x}) - f(x^*) \leq \varepsilon$ and $g(\hat{x}) \leq \varepsilon$. 

To solve problem \eqref{general_problem}, we propose Algorithm \ref{alg_new_constraints} which represents a generalization of Algorithm \ref{alg_mirror_descent} to solve the general problem \eqref{general_problem}. 

The output point (output of Algorithm \ref{alg_new_constraints}) will be selected among the points $x^k$ for which $g(x^k) \leq \varepsilon$. Therefore, we will call step $k$ productive if $g(x^k) \leq \varepsilon$. If the reverse inequality $g(x^k) > \varepsilon$ holds then step $k$ will be called non-productive. Let $I$ and $J$ denote the set of productive and non-productive steps, respectively. $|I|$ and $|J|$ denote the number of productive and non-productive steps, respectively. Let us also set $\gamma_k := \gamma_k^f$ if $k \in I$, $\gamma_k := \gamma_k^g$ if $k \in J$, and $M = \max\{M_f, M_g\}$, where $M_g = \max_{1 \leq i \leq p} \{M_{g_i}\}$.

\begin{algorithm}[!ht]
\caption{Mirror descent method for problems with functional constraints.}\label{alg_new_constraints}
\begin{algorithmic}[1]
\REQUIRE $\varepsilon>0$, initial point $x^1  \in Q$, step sizes $\{\gamma_k^f\}_{k \geq 1}, \{\gamma_k^g\}_{k \geq 1}$, number of iterations $N$.
\STATE $I \longrightarrow \emptyset, J \longrightarrow \emptyset. $
\FOR{$k= 1, 2, \ldots, N$}
\IF{$g(x^k) \leq \varepsilon$}
\STATE Calculate $\nabla f(x^k) \in \partial f(x^k)$,
\STATE $x^{k+1} = \arg\min_{x \in Q} \left\{ \langle x, \nabla f(x^k)  \rangle + \frac{1}{\gamma_k^f} V_{\psi}(x,x^k) \right\} $.
\STATE  $k \longrightarrow I$  \qquad "productive step"
\ELSE
\STATE Calculate $\nabla g(x^k) \in \partial g(x^k)$,
\STATE $x^{k+1} = \arg\min_{x \in Q} \left\{ \langle x, \nabla g(x^k)  \rangle + \frac{1}{\gamma_k^g} V_{\psi}(x,x^k) \right\} $.
\STATE  $k \longrightarrow J$  \qquad "non-productive step"
\ENDIF
\ENDFOR
%\ENSURE $\hat{x} = \frac{1}{k} \sum_{i = 1}^{k} x^i$.
\end{algorithmic}
\end{algorithm}

For Algorithm \ref{alg_new_constraints}, we have the following result.

\begin{theorem}\label{theorem_alg3}
Let $f$ and $g$ are Lipschitz convex functions with constants $M_f>0$ and $M_g >0$, respectively. Assume that $\|\nabla f(x)\|_* \leq M_f$ and $\|\nabla g(x)\|_* \leq M_g$, for all $x \in Q$. Then for problem \eqref{general_problem}, with positive non-increasing sequences of step sizes $\{\gamma_k^f\}_{k \geq 1}, \{\gamma_k^g\}_{k \geq 1}$, for any fixed $m \geq -1$, and $\varepsilon > 0 $,  after $N$ iterations of Algorithm \ref{alg_new_constraints}, it satisfies the following inequality
\begin{equation}\label{main_ineq_alg3}
\begin{aligned}
f(\hat{x}) - f(x^*)  & <  \frac{1}{\sum_{k \in I}   \left(\gamma_k^f\right)^{-m}} \Bigg(  \frac{V_{\psi}(x^*, x^1) }{\gamma_N^{m+1}} 
  + \frac{1}{2 \sigma} \sum_{k \in I} \frac{\|\nabla f(x^k)\|_*^2}{\left(\gamma_k^{f}\right)^{m-1}}
 \\&  \qquad \qquad \qquad \qquad \quad +  \frac{1}{2 \sigma} \sum_{k \in J} \frac{\|\nabla g(x^k)\|_*^2}{\left(\gamma_k^{g}\right)^{m-1}} {\color{black}  - \varepsilon \sum_{k \in J} \left(\gamma_k^g\right)^{-m}  } \Bigg),
\end{aligned}
\end{equation}
where $\hat{x} = \frac{1}{\sum\limits_{k \in I} \left(\gamma_k^f\right)^{-m}} \sum_{k \in I} \left(\gamma_k^f\right)^{-m} x^k $,  and $x^*$ is an optimal solution of \eqref{general_problem}.
\end{theorem}
\begin{proof}
Similar to what was done in proving Theorem \ref{theo_main_ineq_mirror_desc}, we find that for all $k \in I$, 
\begin{equation}\label{ineq_productive}
\frac{f(x^k) - f(x^*)}{\left(\gamma_k^f\right)^m}  \leq     \frac{V_{\psi}(x^*, x^k) -V_{\psi}(x^*, x^{k+1}) }{\left(\gamma_k^f\right)^{m + 1}}  + \frac{\|\nabla f(x^k)\|_*^2}{2 \sigma \left(\gamma_k^f\right)^{m - 1}},
\end{equation}
and for all $k \in J$, 
\begin{equation}\label{ineq_nonproductive}
\frac{g(x^k) - g(x^*)}{\left(\gamma_k^g\right)^m}  \leq     \frac{V_{\psi}(x^*, x^k) -V_{\psi}(x^*, x^{k+1}) }{\left(\gamma_k^g\right)^{m + 1}}  + \frac{\|\nabla g(x^k)\|_*^2}{2 \sigma \left(\gamma_k^g\right)^{m - 1}}.
\end{equation}

By taking the summation, in each side of \eqref{ineq_productive} and \eqref{ineq_nonproductive}, over productive and non-productive steps, with $\gamma_k = \gamma_k^f$ if $k \in I$ and $\gamma_k = \gamma_k^g$ if $k \in J$, we get
$$
\begin{aligned}
    & \sum_{k \in I} \left(\gamma_k^f\right)^{-m} \left( f(x^k) - f(x^*) \right) + \sum_{k \in J} \left(\gamma_k^g\right)^{-m} \left( g(x^k) - g(x^*) \right)  
\\& \leq \sum_{k = 1}^{N}\frac{1}{\gamma_k^{m+1}} \left( V_{\psi} (x^*, x^{k}) - V_{\psi} (x^*, x^{k+1}) \right) +\frac{1}{2 \sigma} \sum_{k \in I} \frac{\|\nabla f(x^k)\|_*^2}{ \left(\gamma_k^f\right)^{m - 1}}
    \\& \;\;\;\; + \frac{1}{2 \sigma} \sum_{k \in J} \frac{\|\nabla g(x^k)\|_*^2}{ \left(\gamma_k^g\right)^{m - 1}}.
\end{aligned}
$$

Also, in a similar way as in the proof of Theorem \ref{theo_main_ineq_mirror_desc}, and since for any $k \in J$, it holds
\begin{equation}\label{eq_nonprod}
g(x^k) - g(x^*) \geq g(x^k) > \varepsilon > 0,
\end{equation}
we get
$$
\begin{aligned}
     \sum_{k \in I} \left(\gamma_k^f\right)^{-m} \left( f(x^k) - f(x^*) \right) & < \frac{1}{\gamma_N^{m+1}}V_{\psi} (x^*, x^1)+     \frac{1}{2 \sigma} \sum_{k \in I} \frac{\|\nabla f(x^k)\|_*^2}{ \left(\gamma_k^f\right)^{m - 1}} 
    \\&  \;\;\;\;  +  \frac{1}{2 \sigma} \sum_{k \in J} \frac{\|\nabla g(x^k)\|_*^2}{ \left(\gamma_k^g\right)^{m - 1}} {\color{black} - \varepsilon \sum_{k \in J} \left(\gamma_k^g\right)^{-m}  }.
\end{aligned}
$$

%As long as the inequality is strict, then $I \ne \emptyset$, and the point $\hat{x}$ is well-defined.
Since $f$ is a convex function, we have 
$$
\begin{aligned}
    \left(\sum_{k \in I} \left(\gamma_k^f\right)^{-m}\right) & \left[ f\left(\frac{1}{\sum_{k \in I} \left(\gamma_k^f\right)^{-m}} \sum_{k \in I} \left(\gamma_k^f\right)^{-m} x^k\right) - f(x^*) \right] \leq
    \\& \qquad \qquad \qquad \qquad \qquad \leq \sum_{k \in I} \left(\gamma_k^f\right)^{-m} \left( f(x^k) - f(x^*) \right) . 
\end{aligned}
$$

Therefore, by setting $\hat{x} := \frac{1}{\sum\limits_{k \in I} \left(\gamma_k^f\right)^{-m}} \sum_{k \in I} \left(\gamma_k^f\right)^{-m} x^k $, we get
\begin{equation}\label{eqqqq}
\begin{aligned}
\left(\sum_{k \in I} \left(\gamma_k^f\right)^{-m}\right)  \left( f(\hat{x}) - f(x^*) \right) & < \frac{1}{\gamma_N^{m+1}}V_{\psi} (x^*, x^1)+     \frac{1}{2 \sigma} \sum_{k \in I} \frac{\|\nabla f(x^k)\|_*^2}{ \left(\gamma_k^f\right)^{m - 1}} 
\\&  \;\;\;\;  +  \frac{1}{2 \sigma} \sum_{k \in J} \frac{\|\nabla g(x^k)\|_*^2}{ \left(\gamma_k^g\right)^{m - 1}} {\color{black} - \varepsilon \sum_{k \in J} \left(\gamma_k^g\right)^{-m} }.
\end{aligned}
\end{equation}

By dividing both sides of \eqref{eqqqq} by $\sum_{k \in I} \left(\gamma_k^f\right)^{-m}$, we get the desired inequality \eqref{main_ineq_alg3}. 

\bigskip 

{\color{black}
Now, let us show that $|I| \ne 0$. For this, let us assume that $|I| = 0$, therefore $|J| = N$,  i.e., all steps are non-productive. 
From \eqref{eq_nonprod}, we get
\begin{equation}\label{refdg12}
\sum_{k = 1}^{N} \frac{g(x^k) - g(x^*)}{\gamma_{k}^m} > \sum_{k = 1}^{N} \frac{\varepsilon}{\gamma_{k}^m} = \frac{\varepsilon M^m}{\left(\sqrt{2 \sigma}\right)^{m}} \sum_{k = 1}^{N} \left(\sqrt{k}\right)^m, 
\end{equation}
and for all $k \in J = \{1, \ldots, N\}$ with the assumption that $V_{\psi}(x^*, x^1) \leq \theta$ for some $\theta >0$,  we get
\begin{equation*}
\begin{aligned}
\sum_{k = 1}^{N} \frac{g(x^k) - g(x^*)}{\gamma_{k}^m}  & \leq \frac{\theta}{\gamma_N^{m+1}} + \frac{1}{2 \sigma} \sum_{k = 1}^{N} \frac{\|\nabla g(x^k)\|_*^2}{\gamma_k^{m-1}}
\\& \leq \frac{M^{m+1}}{\left(\sqrt{2 \sigma}\right)^{m+1}} \left(\theta \left(\sqrt{N}\right)^{m+1} + \sum_{k = 1}^{N} \left(\sqrt{k}\right)^{m-1} \right).
\end{aligned}
\end{equation*}

But,  one can numerically verify that for a sufficiently big number of iterations $N$ (dependently on suitable values of the parameters $\theta > 0, m \geq -1, M> 0, \varepsilon > 0, \sigma > 0 $),  the following inequality holds
\begin{equation}\label{eq_nnnn}
    \frac{M^{m+1}}{\left(\sqrt{2 \sigma}\right)^{m+1}} \left(\theta \left(\sqrt{N}\right)^{m+1} + \sum_{k = 1}^{N} \left(\sqrt{k}\right)^{m-1} \right) < \frac{\varepsilon M^m}{\left(\sqrt{2 \sigma}\right)^{m}} \sum_{k = 1}^{N} \left(\sqrt{k}\right)^m. 
\end{equation}
Therefore, we get 
$$
\sum_{k = 1}^{N} \frac{g(x^k) - g(x^*)}{\gamma_{k}^m}  < \frac{\varepsilon M^m}{\left(\sqrt{2 \sigma}\right)^{m}} \sum_{k = 1}^{N} \left(\sqrt{k}\right)^m.
$$
So, we have a contradiction with \eqref{refdg12}. This means that $|I| \ne 0$. 
} 
\end{proof}

{\color{black} 
\begin{remark}
Note that the reverse inequality of \eqref{eq_nnnn}, i.e., 
$$
\sum_{k = 1}^{N} \left(\sqrt{k}\right)^m \leq \frac{M}{\varepsilon \sqrt{2 \sigma}} \left(\theta \left(\sqrt{N}\right)^{m+1} + \sum_{k = 1}^{N} \left(\sqrt{k}\right)^{m-1}\right), 
$$
for any $m \geq -1, M>0, \theta >0, \sigma > 0$ and $\varepsilon \leq \frac{M}{\sqrt{2 \sigma}}$,   holds for at least $N = 1$. This means that by choosing $\varepsilon \leq \frac{M}{\sqrt{2 \sigma}} \, (\forall M>, \sigma > 0)$, we have at least one productive step for any $m \geq -1$ and $\theta > 0 $.   
\end{remark}
}

}

%%% this remark for the stopping criterion of Algorithm 3
\begin{remark}[Stopping criterion for Algorithm \ref{alg_new_constraints}]
Let $x^1 \in Q$ be an initial point of Algorithm \ref{alg_new_constraints}, such that $V_{\psi}(x^*, x^1) \leq \theta_1$ for some $\theta_1 >0$. From inequality \eqref{eqqqq} (which is equivalent to the inequality \eqref{main_ineq_alg3}),  with $\gamma_q = \gamma_q^f$ if $q \in I$ and $\gamma_q = \gamma_q^g$ if $q \in J$,  with 
\begin{equation}\label{out_hatx}
    \hat{x} = \frac{1}{\sum_{i \in I} \left(\gamma_i^f\right)^{-m}} \sum_{i \in I} \left(\gamma_i^f\right)^{-m} x^i,
\end{equation}
for any $k \geq 1$, we find 
\begin{equation*}
\begin{aligned}
& \left(\sum_{i \in I} \left(\gamma_i^f\right)^{-m}\right)  \left( f(\hat{x}) - f(x^*) \right) 
< 
\frac{1}{\gamma_k^{m+1}}V_{\psi} (x^*, x^1)+     \frac{1}{2 \sigma} \sum_{i \in I} \frac{\|\nabla f(x^i)\|_*^2}{ \left(\gamma_i^f\right)^{m - 1}} 
\\& \qquad \qquad \qquad \qquad  +  \frac{1}{2 \sigma} \sum_{j \in J} \frac{\|\nabla g(x^j)\|_*^2}{ \left(\gamma_j^g\right)^{m - 1}}  - \varepsilon \sum_{i = 1}^{k} \left(\gamma_i\right)^{-m}  + \varepsilon \sum_{i \in I} \left(\gamma_i^f\right)^{-m} 
\\& \quad \qquad 
\leq  \varepsilon \sum_{i \in I} \left(\gamma_i^f\right)^{-m} - \Bigg( \varepsilon \sum_{i = 1}^{k} \left(\gamma_i\right)^{-m} - \frac{\theta_1}{\gamma_k^{m+1}} - \frac{1}{2 \sigma}\sum_{i \in I} \frac{\|\nabla f(x^i)\|_*^2}{ \left(\gamma_i^f\right)^{m - 1}} 
\\& \qquad \qquad \qquad \qquad \qquad \qquad \qquad \qquad \qquad \qquad \quad \; 
-   \frac{1}{2 \sigma}\sum_{j \in J} \frac{\|\nabla g(x^j)\|_*^2}{ \left(\gamma_j^g\right)^{m - 1}}
\Bigg).
\end{aligned}
\end{equation*}

From this, without relying on prior knowledge of the number of iterations that the algorithm performs, we can set for any $k \geq 1$,
\begin{equation}\label{stop_criter_alg3}
\varepsilon \sum_{i = 1}^{k} \frac{1}{\gamma_i^m} \geq 
\frac{\theta_1}{\gamma_k^{m+1}} + \frac{1}{2 \sigma}\sum_{i \in I} \frac{\|\nabla f(x^i)\|_*^2}{ \left(\gamma_i^f\right)^{m - 1}} + \frac{1}{2 \sigma}\sum_{j \in J} \frac{\|\nabla g(x^j)\|_*^2}{ \left(\gamma_j^g\right)^{m - 1}},
\end{equation}
as a stopping criterion of Algorithm \ref{alg_new_constraints}. As a result, we conclude 
$$   
\left(\sum_{i \in I} \left(\gamma_i^f\right)^{-m}\right)  \left( f(\hat{x}) - f(x^*) \right) <   \varepsilon \sum_{i \in I} \left(\gamma_i^f\right)^{-m}, 
$$
i.e., $f(\hat{x}) - f(x^*) \leq \varepsilon$. 

For all $i \in I$ it holds that $g(x^i) \leq \varepsilon$, and since $g$ is convex, then we have 
$$
g(\hat{x}) \leq \frac{1}{\sum_{i \in I} \left(\gamma_i^f\right)^{-m}} \sum_{i \in I} \left(\gamma_i^f\right)^{-m} g(x^i) \leq \varepsilon.
$$

Thus after the stopping criterion \eqref{stop_criter_alg3} is met we find that $\hat{x}$, which is given in \eqref{out_hatx}, represents an $\varepsilon$-solution to problem \eqref{general_problem}. 
\end{remark}

Now let us analyze the convergence of Algorithm \ref{alg_new_constraints}, by taking the following time-varying step size rules
\begin{equation}\label{steps_rulse_alg3}
\gamma_k = 
\begin{cases}
\gamma_k^f : = \frac{\sqrt{2 \sigma}}{M_f \sqrt{k}}, \quad  \text{if} \; k \in I, \\
\gamma_k^g : = \frac{\sqrt{2 \sigma}}{M_g \sqrt{k}}, \quad  \text{if} \; k \in J.
\end{cases}
\end{equation}

Let us assume that $V_{\psi} (x^*, x^1) \leq \theta_1, $ for some $\theta_1 > 0$, $M: = \max \{M_f, M_g\}$. By using \eqref{steps_rulse_alg3}, and since $\|\nabla f(x^k)\|_* \leq M_f \leq M$ and $\|\nabla g(x^k)\|_* \leq M_g \leq M$, then for any $m \geq 1 $, from Theorem \ref{theorem_alg3} we have
$$
\begin{aligned}
f(\hat{x}) - f(x^*)  & < \frac{\left(\sqrt{2\sigma}\right)^{m}}{M^m \sum_{k \in I} \left(\sqrt{k}\right)^{m}} \Bigg(  \frac{\theta_1 M^{m+1} \left(\sqrt{N}\right)^{m+1} }{\left(\sqrt{2\sigma}\right)^{m+1}} + 
\\& \qquad \qquad \qquad \qquad \qquad \qquad + \frac{1}{2 \sigma} \sum_{k = 1}^{N} \frac{ M^{m+1} \left(\sqrt{k}\right)^{m-1} }{\left(\sqrt{2\sigma}\right)^{m-1}} \Bigg)
\\& = \frac{M}{\sqrt{2\sigma}} \cdot \frac{1}{\sum_{k \in I} \left(\sqrt{k}\right)^m} \cdot \left(\theta_1 \left(\sqrt{N}\right)^{m+1} + \sum_{k = 1}^{N} \left(\sqrt{k}\right)^{m-1} \right) 
\\& \leq \frac{M}{\sqrt{2 \sigma}} \cdot\frac{1}{\sum_{k \in I} \left(\sqrt{k}\right)^m} \cdot \left( \theta_1 \left(\sqrt{N}\right)^{m+1} + N \left(\sqrt{N}\right)^{m-1} \right)
\\& = \frac{M(1+ \theta_1) \left(\sqrt{N}\right)^{m+1}}{\sqrt{2 \sigma}} \cdot\frac{1}{\sum_{k \in I} \left(\sqrt{k}\right)^m}. 
\end{aligned}
$$

Now, by setting $\frac{M(1+ \theta_1) \left(\sqrt{N}\right)^{m+1}}{\sqrt{2 \sigma}} \cdot\frac{1}{\sum_{k \in I} \left(\sqrt{k}\right)^m} \leq \varepsilon$ and since $|I| \leq N$, we get 
$$
\frac{M(1+ \theta_1) \left(\sqrt{N}\right)^{m+1}}{\sqrt{2 \sigma} N\left(\sqrt{N}\right)^m} \leq \frac{M(1+ \theta_1) \left(\sqrt{N}\right)^{m+1}}{\sqrt{2 \sigma}} \cdot\frac{1}{\sum_{k \in I} \left(\sqrt{k}\right)^m} \leq  \varepsilon.
$$
Thus, 
$$
\frac{M (1 + \theta_1)}{\sqrt{N} \sqrt{2\sigma}}  \leq \varepsilon \quad \Longrightarrow \quad N \geq \frac{M^2 (1 + \theta_1)^2}{2 \sigma \varepsilon^2}. 
$$
%We also can conclude the same results if we take the adaptive step size rules given in \eqref{steps_rulse_alg3}.

Hence, we can formulate the following result.
\begin{corollary}
Let $f$ and $g$ are Lipschitz convex functions with constants $M_f>0$ and $M_g >0$, respectively. Assume that $\|\nabla f(x)\|_* \leq M_f$ and $\|\nabla g(x)\|_* \leq M_g$, for all $x \in Q$, and $V_{\psi}(x^*, x^1) \leq \theta_1, $ for some $\theta_1 >0$, and let $\varepsilon >0, M := \max \{M_f, M_g\}$. Then after 
\begin{equation}\label{iter_bound_alg3}
    N = \left\lceil \frac{M^2 (1+  \theta_1)^2}{2 \sigma \varepsilon^2} \right\rceil = O \left(\frac{1}{\varepsilon^2}\right)
\end{equation}
iterations of Algorithm \ref{alg_new_constraints}, for any fixed $m \geq 1$, with step size rules given in \eqref{steps_rulse_alg3}, it satisfies 
$$
f(\hat{x}) - f(x^*)< \varepsilon, \quad \text{and} \quad g(\hat{x}) \leq \varepsilon, 
$$
where $\hat{x} = \frac{1}{\sum_{k \in I} \left(\gamma_k^f\right)^{-m}} \sum\limits_{k \in I} \left(\gamma_k^f\right)^{-m} x^k$.
\end{corollary}

The estimate \eqref{iter_bound_alg3}, is optimal for the class of non-smooth optimization problems under consideration. 

%\newpage
%Then we can deduce the optimal convergence rate of Algorithm \ref{alg_new_constraints}, for the considered class non-smooth problems \eqref{general_problem_1} (or its equivalence \eqref{general_problem}). 

\bigskip 

Now, by setting $m = 0$ in \eqref{main_ineq_alg3}, with $\overline{x} = \frac{1}{|I|} \sum_{k \in I} x^k$, we get
$$
\begin{aligned}
f(\overline{x}) - f(x^*) & < \frac{1}{|I|} \left(\frac{V_{\psi}(x^*, x^1)}{\gamma_N} + \sum_{k \in I} \frac{\|\nabla f(x^k)\|_*^2 \gamma_k^f}{2 \sigma} + \sum_{k \in J} \frac{\|\nabla g(x^k)\|_*^2 \gamma_k^g}{2 \sigma}    \right) ,
\\&  \leq \frac{1}{|I|} \left( \frac{M \theta_1 \sqrt{N}}{\sqrt{2 \sigma}} + \frac{M}{\sqrt{2 \sigma}} \sum_{k \in I} \frac{1}{\sqrt{k}} +\frac{M}{\sqrt{2 \sigma}} \sum_{k \in J} \frac{1}{\sqrt{k}}  \right)
\\& = \frac{M}{|I| \sqrt{2\sigma}} \left(\theta_1 \sqrt{N} + \sum_{k = 1}^{N} \frac{1}{\sqrt{k}}\right)  \leq \frac{M}{|I| \sqrt{2\sigma}} \left(\theta_1 \sqrt{N} + 2 \sqrt{N}\right)
\\& = \frac{M \sqrt{N} (2 + \theta_1)}{|I| \sqrt{2\sigma}} .
\end{aligned}
$$

Now, by setting $\frac{M \sqrt{N} (2 + \theta_1)}{|I| \sqrt{2\sigma}}  \leq \varepsilon$ and since $|I| \leq N$, we get 
$$
\frac{M (2 + \theta_1)}{\sqrt{N} \sqrt{2\sigma}}  \leq \varepsilon \quad \Longrightarrow \quad N \geq \frac{M^2 (2 + \theta_1)^2}{2 \sigma \varepsilon^2}.
$$

Hence, for $m=0$,  we can formulate the following result.

\begin{corollary}
Let $f$ and $g$ are Lipschitz convex functions with constants $M_f>0$ and $M_g >0$, respectively. Assume that $\|\nabla f(x)\|_* \leq M_f$ and $\|\nabla g(x)\|_* \leq M_g$, for all $x \in Q$, and $V_{\psi}(x^*, x^1) \leq \theta_1, $ for some $\theta_1 >0$, and let $\varepsilon >0, M := \max \{M_f, M_g\}$. Then after 
$$
    N = \left\lceil \frac{M^2 (2+  \theta_1)^2}{2 \sigma \varepsilon^2} \right\rceil = O \left(\frac{1}{\varepsilon^2}\right),
$$
iterations of Algorithm \ref{alg_new_constraints}, with  $m = 0$, and step size rules given in \eqref{steps_rulse_alg3}, it satisfies 
$$
    f(\overline{x}) - f(x^*) < \varepsilon, \quad \text{and} \quad g(\overline{x}) \leq \varepsilon, 
$$
where $\overline{x} = \frac{1}{|I|} \sum_{k \in I} x^k$.
\end{corollary}

\begin{remark}
By setting $m = -1$ in  \eqref{main_ineq_alg3}, we deduce the sub-optimal convergence rate for Algorithm \ref{alg_new_constraints}. In this case, with $\tilde{x} = \frac{1}{\sum_{k \in I} \frac{1}{\sqrt{k}}} \sum_{k \in I}\frac{1}{\sqrt{k}} x^k$, we have
$$
\begin{aligned}
f(\tilde{x}) - f(x^*) & \leq \frac{M}{\sqrt{2 \sigma}} \cdot \frac{1}{\sum_{k \in I} \sqrt{k}} \cdot \left( \theta_1 + \sum_{k = 1}^{N} \frac{1}{k}\right)
\\& \leq\frac{M}{\sqrt{2 \sigma}} \cdot \frac{1}{2 \sqrt{|I| + 1 } - 2} \cdot \left( \theta_1 + 1 + \log (N)\right)
\\& \leq \frac{M \left( \theta_1 + 1 + \log (N)\right) }{ \sqrt{\sigma} \sqrt{|I|}}. 
\end{aligned}
$$
The presence of the logarithm in the last inequalities is what causes the sub-optimality in the convergence of Algorithm \ref{alg_new_constraints} for $m = -1$. 
\end{remark}

\subsection{Modification of mirror descent method for problems with many functional constraints}

For solving problem \eqref{general_problem_1}, when we have several functional constraints, we can modify Algorithm \ref{alg_new_constraints} in a similar way as in \cite{stonyakin2018}. When we have a non-productive step $k$, i.e., it holds that $g(x^k) > \varepsilon$,  then instead of calculating a (sub)gradient of the functional constraint with max-type $g(x) = \max_{1 \leq i \leq p} \{g_i(x)\}$, we calculate a (sub)gradient of one functional $g_i$, for which we have $g_i(x^k) > \varepsilon$. This idea of the proposed modification allows for saving the running time of the algorithm due to consideration of not all functional constraints (when their number is big) on non-productive steps. The proposed modification is listed as Algorithm \ref{alg_new_constraints_modif} below.

\begin{algorithm}[htp]
\caption{Modification of mirror descent method for problems with many functional constraints.}\label{alg_new_constraints_modif}
\begin{algorithmic}[1]
\REQUIRE $\varepsilon>0$, initial point $x^1  \in Q$, $M = \max \{M_f, M_g\}$, $\theta_1 > 0$ such that $V_{\psi} (x^*, x^1) \leq \theta_1$, $m \geq -1$, $\sigma > 0$.  
\STATE $I \longrightarrow \emptyset, J \longrightarrow \emptyset. $
\STATE Set $k = 0$
\REPEAT
\IF{$g_i(x^k) \leq \varepsilon, \; \forall i = 1, 2, \ldots, p$}
\STATE $L_k = \|\nabla f(x^k)\|_*,$
\STATE $\gamma_k^f = \frac{\sqrt{2 \sigma}}{L_k \sqrt{k}},$
\STATE $x^{k+1} = \arg\min_{x \in Q} \left\{ \langle x, \nabla f(x^k)  \rangle + \frac{1}{\gamma_k^f} V_{\psi}(x,x^k) \right\} ,$
\STATE  $k \longrightarrow I$  \qquad "productive step"
\ELSE
\STATE{(i.e., $\exists q = q(k) \in \{1, 2, \ldots, p\}$, s.t., $g_{q(k)} > \varepsilon$)}
\STATE $L_k = \|\nabla g_{q(k)}(x^k)\|_*,$
\STATE $\gamma_k^{g_{q(k)}} = \frac{\sqrt{2 \sigma}}{L_k \sqrt{k}}, $
\STATE $x^{k+1} = \arg\min_{x \in Q} \left\{ \langle x, \nabla g_{q(k)}(x^k)  \rangle + \frac{1}{\gamma_k^{g_{q(k)}}} V_{\psi}(x, x^k) \right\}, $
\STATE  $k \longrightarrow J$  \qquad "non-productive step"
\ENDIF
\UNTIL{
\begin{align*}
    \varepsilon \sum_{i = 1}^{k} \left(\frac{L_i \sqrt{i}}{\sqrt{2\sigma}}\right)^m \geq \theta_1 \left(\frac{M \sqrt{k}}{\sqrt{2 \sigma}}\right)^{m+1} &+ \frac{1}{\left(\sqrt{2\sigma}\right)^{m+1}}  \Bigg(\sum_{i \in I} \left(\sqrt{i}\right)^{m-1} \|\nabla f(x^i)\|_*^{m+1} 
    \\& \qquad \qquad + \sum_{j \in J}\left(\sqrt{j}\right)^{m-1} \|\nabla g_{q(j)}(x^j)\|_*^{m+1}  \Bigg).
\end{align*}
}
\ENSURE $\hat{x} = \frac{1}{\sum_{i \in I} \left(\gamma_i^f\right)^{-m}} \sum\limits_{i \in I} \left(\gamma_i^f\right)^{-m} x^i. $
\end{algorithmic}
\end{algorithm}

For Algorithm \ref{alg_new_constraints_modif}, we have the following result.

\begin{theorem}\label{theorem_alg4}
Let $f$ and $g_i \, (i = 1, 2, \ldots, p)$ are Lipschitz convex functions with constants $M_f>0$ and $M_{g_i} >0 \, (i = 1, 2, \ldots, p)$, respectively. Assume that, for all $x \in Q$, we have $\|\nabla f(x)\|_* \leq M_f$ and $\|\nabla g_i(x)\|_* \leq M_g \, \forall i = 1, 2, \ldots, p$, where $M_g = \max_{1 \leq i \leq p} \{M_{g_i}\}$, and $V_{\psi} (x^*, x^1) \leq \theta_1$, for some $\theta_1 > 0$, and let $\varepsilon >0$, $M = \max \{M_f, M_g\}$. Then for problem \eqref{general_problem_1}, after stopping Algorithm \ref{alg_new_constraints_modif}, for any fixed $m \geq -1$, it satisfies the following inequalities
\begin{equation}\label{main_ineq_alg4}
f(\hat{x}) - f(x^*) < \varepsilon, \quad \text{and} \quad g_i(\hat{x}) \leq \varepsilon, \, \forall i = 1, \ldots, p. 
\end{equation}
where $\hat{x} = \frac{1}{\sum_{i \in I} \left(\gamma_i^f\right)^{-m}} \sum\limits_{i \in I} \left(\gamma_i^f\right)^{-m} x^i $,  and $x^*$ is an optimal solution of \eqref{general_problem_1}.
\end{theorem}
\begin{proof}

Similar to what was done in proving Theorem \ref{theorem_alg3}, we find that for all $i \in I$, 
\begin{equation}\label{ineq_productive_alg4}
\frac{f(x^i) - f(x^*)}{\left(\gamma_i^f\right)^m}  \leq     \frac{V_{\psi}(x^*, x^i) -V_{\psi}(x^*, x^{i+1}) }{\left(\gamma_i^f\right)^{m + 1}}  + \frac{\|\nabla f(x^i)\|_*^2}{2 \sigma \left(\gamma_i^f\right)^{m - 1}},
\end{equation}
and for all $j \in J$, there is $q(j) \in \{1, 2, \ldots, p\}$, such that
\begin{equation}\label{ineq_nonproductive_alg4}
\frac{g(x^j) - g(x^*)}{\left(\gamma_j^{g_{q(j)}}\right)^m}  \leq     \frac{V_{\psi}(x^*, x^j) -V_{\psi}(x^*, x^{j+1}) }{\left(\gamma_j^{g_{q(j)}}\right)^{m + 1}}  + \frac{\|\nabla g_{q(j)}(x^j)\|_*^2}{2 \sigma \left(\gamma_j^{g_{q(j)}}\right)^{m - 1}}.
\end{equation}

By taking the summation, in each side of \eqref{ineq_productive_alg4} and \eqref{ineq_nonproductive_alg4}, over productive and non-productive steps, with $\gamma_i = \gamma_i^f$ if $i \in I$ and $\gamma_j = \gamma_j^{g_{q(j)}}$ if $i \in J$, we get
\begin{align*}
    & \quad \sum_{i \in I} \left(\gamma_i^f\right)^{-m} \left( f(x^i) - f(x^*) \right) + \sum_{j \in J} \left(\gamma_j^{g_{q(j)}}\right)^{-m} \left( g_{q(j)}(x^j) - g_{q(j)}(x^*) \right)  
    \\& \leq \sum_{i = 1}^{k}\frac{1}{\gamma_i^{m+1}} \left( V_{\psi} (x^*, x^{i}) - V_{\psi} (x^*, x^{i+1}) \right) +\frac{1}{2 \sigma} \sum_{i \in I} \frac{\|\nabla f(x^i)\|_*^2}{ \left(\gamma_i^f\right)^{m - 1}}
    \\& \;\;\;\; + \frac{1}{2 \sigma} \sum_{j \in J} \frac{\|\nabla g_{q(j)}(x^j)\|_*^2}{ \left(\gamma_j^{g_{q(j)}}\right)^{m - 1}}.
\end{align*}

Since for any $j \in J$, it holds
\begin{equation*}
g_{q(j)}(x^j) - g_{q(j)}(x^*) \geq g_{q(j)}(x^j) > \varepsilon > 0,
\end{equation*}
and since $f$ is a convex function, we get
\begin{align*}
& \quad \left(\sum_{i \in I} \left(\gamma_i^f\right)^{-m}\right)  \left( f(\hat{x}) - f(x^*) \right) 
\\&  < \frac{\theta_1}{\gamma_k^{m+1}} + \frac{1}{2 \sigma} \sum_{i \in I} \frac{\|\nabla f(x^i)\|_*^2}{ \left(\gamma_i^f\right)^{m - 1}}  +  \frac{1}{2 \sigma} \sum_{j \in J} \frac{\|\nabla g_{q(j)}(x^j)\|_*^2}{ \left(\gamma_j^{g_{q(j)}}\right)^{m - 1}}  - \varepsilon \sum_{j \in J} \left(\gamma_j^{g_{q(j)}}\right)^{-m} 
\\&   = \frac{\theta_1}{\gamma_k^{m+1}} + \frac{1}{2 \sigma} \sum_{i \in I} \frac{\|\nabla f(x^i)\|_*^2}{ \left(\gamma_i^f\right)^{m - 1}}  +  \frac{1}{2 \sigma} \sum_{j \in J} \frac{\|\nabla g_{q(j)}(x^j)\|_*^2}{ \left(\gamma_j^{g_{q(j)}}\right)^{m - 1}}  - \varepsilon \sum_{i = 1}^{k} (\gamma_i)^{-m}
\\& \quad  + \varepsilon \sum_{i \in I} \left(\gamma_i^{f}\right)^{-m}.
\end{align*}
But $\gamma_k = \frac{\sqrt{2 \sigma}}{\|\nabla f(x^k)\|_* \sqrt{k}}$ or $\gamma_k = \frac{\sqrt{2 \sigma}}{\|\nabla g_{q(k)}(x^k)\|_* \sqrt{k}}$ and in both cases we have $(\gamma_k)^{m+1} \geq \left(\frac{\sqrt{2\sigma}}{M \sqrt{k}}\right)^{m+1}$. Thus, we have

\begin{align*}
& \quad \left(\sum_{i \in I} \left(\gamma_i^f\right)^{-m}\right)  \left( f(\hat{x}) - f(x^*) \right) 
\\&  < \varepsilon \sum_{i \in I} \left(\gamma_i^{f}\right)^{-m} + \theta_1 \left(\frac{M \sqrt{k}}{\sqrt{2\sigma}}\right)^{m+1}  + \frac{1}{2 \sigma} \sum_{i \in I}\|\nabla f(x^i)\|_*^2\left(\frac{\|\nabla f(x^i)\|_* \sqrt{i}}{\sqrt{2\sigma}}\right)^{m-1} 
\\& \quad +  \frac{1}{2 \sigma} \sum_{j \in J} \|\nabla g_{q(j)}(x^j)\|_*^2 \left(\frac{\|\nabla g_{q(j)}(x^j)\|_* \sqrt{j}}{\sqrt{2\sigma}}\right)^{m-1} - \varepsilon \sum_{i = 1}^{k} \left(\frac{L_i \sqrt{i}}{\sqrt{2 \sigma}}\right)^{m}
\\& \quad = \varepsilon \sum_{i \in I} \left(\gamma_i^{f}\right)^{-m} - \Bigg( \varepsilon \sum_{i = 1}^{k} \left(\frac{L_i \sqrt{i}}{\sqrt{2 \sigma}}\right)^{m} - \theta_1 \left(\frac{M \sqrt{k}}{\sqrt{2\sigma}}\right)^{m+1} 
\\& \qquad \qquad \qquad \qquad \qquad 
- \frac{1}{(\sqrt{2 \sigma})^{m+1}} \sum_{i \in I} (\sqrt{i})^{m-1} \|\nabla f(x^i)\|_*^{m+1}
\\& \qquad \qquad \qquad \qquad \qquad 
- \frac{1}{(\sqrt{2 \sigma})^{m+1}} \sum_{j \in J} (\sqrt{j})^{m-1} \|\nabla g_{q(j)}(x^j)\|_*^{m+1}
\Bigg). 
\end{align*}

Therefore, from stopping criterion of Algorithm \ref{alg_new_constraints_modif}, we get 
$$
\left(\sum_{i \in I} \left(\gamma_i^f\right)^{-m}\right)  \left( f(\hat{x}) - f(x^*) \right)  < \varepsilon \sum_{i \in I} \left(\gamma_i^{f}\right)^{-m}. 
$$
Thus, by dividing both sides of the previous inequality by $\sum_{i \in I} \left(\gamma_i^f\right)^{-m}$, we get the desired inequality $f(\hat{x}) - f(x^*) < \varepsilon$. 

For all $i \in I$ it holds that $g_s(x^i) \leq \varepsilon \, \forall s \in \{1, 2, \ldots, p\}$, and since $g_s \, (\forall s \in \{1, 2, \ldots, p\})$ are convex functions, then we have 
$$
g_s(\hat{x}) \leq \frac{1}{\sum_{i \in I} \left(\gamma_i^f\right)^{-m}} \sum_{i \in I} \left(\gamma_i^f\right)^{-m} g_s(x^i) \leq \varepsilon, \; \forall s \in \{1, 2, \ldots, p\}.
$$
\end{proof}

\begin{remark}
For any $m \geq 1$, let us assume that Algorithm \ref{alg_new_constraints_modif} works $N$ iterations. Then we have
\begin{align*}
& \qquad f(\hat{x}) - f(x^*)   
\\& < \frac{\left(\sqrt{2\sigma}\right)^{m}}{M^m \sum_{k \in I} \left(\sqrt{k}\right)^{m}} \Bigg(  \theta_1 \left(\frac{M \sqrt{N}}{\sqrt{2 \sigma}} \right)^{m+1} + \frac{1}{2 \sigma} \sum_{k = 1}^{N} \frac{M^{m+1} \left(\sqrt{k}\right)^{m-1}}{\left(\sqrt{2 \sigma}\right)^{m-1}} \Bigg)
\\& = \frac{M}{\sqrt{2\sigma}} \cdot \frac{1}{\sum_{k \in I} \left(\sqrt{k}\right)^m} \cdot \left(\theta_1 \left(\sqrt{N}\right)^{m+1} + \sum_{k = 1}^{N} \left(\sqrt{k}\right)^{m-1} \right) 
\\& \leq \frac{M}{\sqrt{2 \sigma}} \cdot\frac{1}{\sum_{k \in I} \left(\sqrt{k}\right)^m} \cdot \left( \theta_1 \left(\sqrt{N}\right)^{m+1} + N \left(\sqrt{N}\right)^{m-1} \right)
\\& = \frac{M(1+ \theta_1) \left(\sqrt{N}\right)^{m+1}}{\sqrt{2 \sigma}} \cdot\frac{1}{\sum_{k \in I} \left(\sqrt{k}\right)^m}.
\end{align*}

Now, by setting $\frac{M(1+ \theta_1) \left(\sqrt{N}\right)^{m+1}}{\sqrt{2 \sigma}} \cdot\frac{1}{\sum_{k \in I} \left(\sqrt{k}\right)^m} \leq \varepsilon$ and since $|I| \leq N$, we get 
$$
\frac{M(1+ \theta_1) \left(\sqrt{N}\right)^{m+1}}{\sqrt{2 \sigma} N\left(\sqrt{N}\right)^m} \leq \frac{M(1+ \theta_1) \left(\sqrt{N}\right)^{m+1}}{\sqrt{2 \sigma}} \cdot\frac{1}{\sum_{k \in I} \left(\sqrt{k}\right)^m} \leq  \varepsilon.
$$
Thus, 
$$
\frac{M (1 + \theta_1)}{\sqrt{N} \sqrt{2\sigma}}  \leq \varepsilon \quad \Longrightarrow \quad N \geq \frac{M^2 (1 + \theta_1)^2}{2 \sigma \varepsilon^2}. 
$$
Therefore, we find that after 
$$
N = \left\lceil \frac{M^2 (1+  \theta_1)^2}{2 \sigma \varepsilon^2} \right\rceil = O \left(\frac{1}{\varepsilon^2}\right)
$$
iterations of Algorithm \ref{alg_new_constraints_modif}, for any fixed $m \geq 1$, it satisfy $f(\hat{x}) - f(x^*) < \varepsilon$ and $g_i (\hat{x}) \leq \varepsilon \, \forall i \in \{1, 2, \ldots, p\}$.  
\end{remark}

\section{Numerical experiments}\label{sect_numerical}
To show the advantages and effects of the weighting scheme for generated points $\hat{x}$, by Algorithm \ref{alg_mirror_descent} (see Theorem \ref{theo_main_ineq_mirror_desc}) (the same will be for the generated points $\hat{x}$ by Algorithm \ref{alg_mirrorC_descent} (see Theorem \ref{theo_main_ineq_mirrorC})), in the convergence of algorithms, a series of numerical experiments were
performed for some problems with a geometrical nature in addition to the problem of minimizing the maximum of a finite collection of linear functions.

First, we mention that for the following adaptive step size scheme
\begin{equation}\label{adaptive_steps}
    \gamma_k = \frac{\sqrt{2 \sigma}}{\|\nabla f(x^k)\|_* \sqrt{k}} \quad k = 1, \ldots N, 
\end{equation}
there is no guarantee that the sequence $\{\gamma_k\}_{k \geq 1}$ is non-increasing. As a result, this step size does not satisfy one of the conditions of Theorem \ref{theo_main_ineq_mirror_desc}. However, Algorithm \ref{alg_mirror_descent} performs well in practice, and the convergence rate remains the same as the step size defined in \eqref{steps_rules}.  Therefore, we need to consider restructuring the scheme of step sizes so that it becomes adaptive and allows us to obtain the same results as previously obtained for the non-adaptive ones.

We compare the performance of the studied algorithms with the projected subgradient method using different famous step size rules that are listed in Table \ref{Tab_steps}. %(which is inspired from \cite{Ferreira2023subgradient}) 
Therefore, in our experiments, we take the standard Euclidean prox-structure, namely $\psi(x) = \frac{1}{2} \|x\|_2^2$ which is $1$-strongly functions (i.e., $\sigma = 1$) and the corresponding Bregman divergence is $V_{\psi}(x, y) = \frac{1}{2} \|x - y\|_2^2$. In all experiments, we take the set $Q$ as a unit ball in $\mathbb{R}^n$ with the center at $0 \in \mathbb{R}^n$. All
compared methods start from the same initial point $x^1 = \left(\frac{1}{\sqrt{n}}, \ldots, \frac{1}{\sqrt{n}}\right) \in Q \subset \mathbb{R}^n$. In the AdaGrad algorithm, we take $\alpha = 10^{-8}$, and there is an assumption that $\|x^1 - x^*\|_2^2 \leq 2 \theta_0^2$, thus for the taken feasible set in our experiments, we can take $\theta_0 = \sqrt{2}$.  

The comparison of the methods is done in terms of the difference $\hat{f}_k - f_{\text{min}}$, where $\hat{f}_k$ denotes the value of the objective function $f$ at the averaged points (namely at $\hat{x}_k = \frac{1}{k} \sum_{j = 1}^{k} x^k$, for all cases in Table \ref{Tab_steps}, except to the case of ''quad grad'', where in this case we have $\hat{x}_k = \left(\sum_{j = 1}^{k} \gamma_k\right)^{-1} \sum_{j=1}^{k} \gamma_k x^k $) and $f_{\text{min}}$ denotes the minimal value of the objective function computed by SciPy, a package for solving many different classes of optimization problems (when the dimension of the space $\mathbb{R}^n$ is not big. Also, we did a comparison of the methods when $\hat{f}_k$ denotes the best value of
$f(x^k)$ attained, i.e., $\hat{f}_k = \min \{f(x^1), \ldots, f(x^k)\}$, see Remarks \ref{fig_best_approx_n1000_outs_mins} and \ref{fig_Fermat_n200_outs_mins}. 

\begin{table}[htp]
\begin{tabular}{l l l }
\hline
Abbreviation & Step sizes  & formula of $\gamma_k$ \\ \hline\hline
constant step & constant step size \cite{Boyd2004Subgradient} & $\gamma_k = 0.1$, \\ \hline
fixed length  & fixed step length \cite{Boyd2004Subgradient} & $\gamma_k = \frac{0.2}{\|\nabla f(x^k)\|_*}$, \\ \hline
nonsum        &
\shortstack{non-summable \\ diminishing step \cite{Boyd2004Subgradient}}    & $\gamma_k = \frac{0.1}{\sqrt{k}}$,\\ \hline
sqrsum nonsum & 
\shortstack{square summable but \\ not summable step \cite{Boyd2004Subgradient}}
  & $\gamma_k = \frac{0.5}{k}$, \\ \hline
quad grad     & \shortstack{quadratic of the norm \\ of the gradient \cite{FedorRolandbook}}  & $\gamma_k = \frac{0.2}{\|\nabla f(x^k)\|_*^2}$, \\ \hline
AdaGrad       & \shortstack{AdaGrad \\ algorithm \cite{Duchi2011Adaptive}}   & $\gamma_k =  \frac{\theta_0}{\sqrt{\sum_{j = 1}^{k} \|\nabla f(x^j)\|_*^2 + \alpha}} $, \\ \hline
Polyak step   & Polyak step size \cite{polyak_book} & $\gamma_k = \frac{f(x^k) - f^*}{\|\nabla f(x^k)\|_*^2}$,  \\ \hline
\shortstack{non-adaptive \\ time-varying} &  scheme \eqref{steps_rules} & $\gamma_k = \frac{\sqrt{2\sigma}}{M_f \sqrt{k}}$, \\ \hline
\shortstack{adaptive \\time-varying} &  scheme \eqref{adaptive_steps} & $\gamma_k = \frac{\sqrt{2\sigma}}{\|\nabla f(x^k)\|_* \sqrt{k}}$. \\ \hline\hline
\end{tabular}
\caption{The used step sizes in the compared methods.}
\label{Tab_steps}
\end{table}

\subsection{Best approximation problem}

The considered problem in this subsection is connected with the problem of the best approximation of the distance between a point and a given set $Q$. For this problem, let $A \notin  Q$ be a given point, we need to solve the following optimization problem
\begin{equation}\label{best_approx_prob}
 \min_{x \in Q}\left\{ f(x) : = \|x - A\|_2 \right\}. 
\end{equation}

The point $A$ is randomly generated from a uniform distribution over $[0, 1)$, such that $\|A\|_2 = 10$, therefore the distance between the point $A$ and the considered unit ball $Q$ is equal to $9$, i.e., $f^* = 9$. Here, we mention that this problem is constructed so that it can be used the Polyak step size, which requires knowing the optimal value $f^*$. The results of the comparison, for problem \eqref{best_approx_prob} with $n = 1000$  are presented in Fig. \ref{fig_best_approx_n1000}. In this figure $\hat{f}_k$  denotes the value of the objective function $f$ at the averaged points in each iteration of all compared algorithms. From this figure, we can see that Algorithm \ref{alg_mirror_descent}, with non-adaptive and adaptive step size rules \eqref{steps_rules}, works better than other algorithms, where the difference between the performance of Algorithm \ref{alg_mirror_descent} and the rest of algorithms with steps in Table \ref{Tab_steps} is clear and significant. We also note that there is no significant difference in the performance of the Algorithm \ref{alg_mirror_descent} when using the adaptive or non-adaptive step size rules steps  

\begin{figure}[htp]
\minipage{0.95\textwidth}
\includegraphics[width=\linewidth]{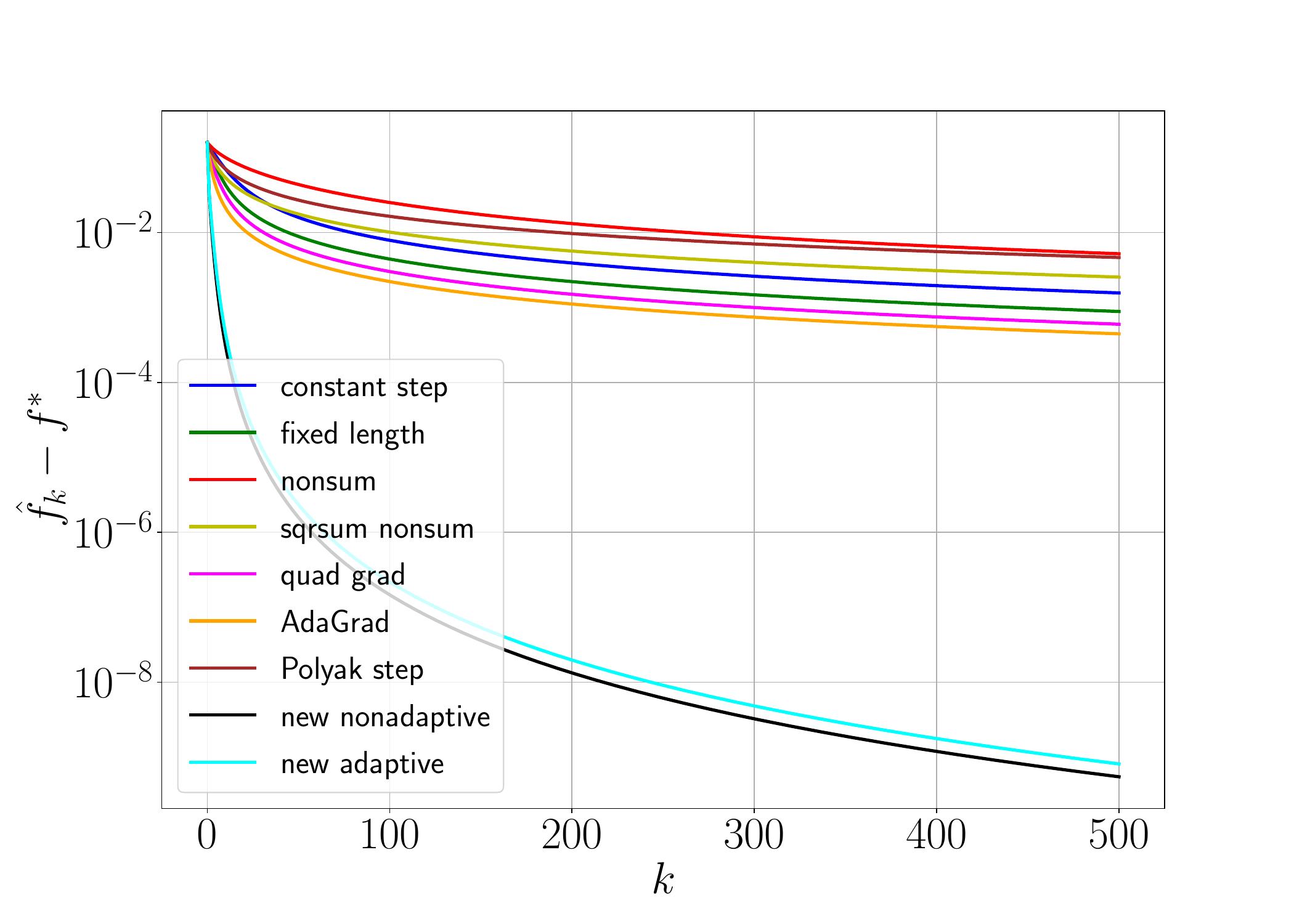}
\endminipage%\hfill
\caption{Results of  Algorithm \ref{alg_mirror_descent} and projected subgradient method using different step size rules listed in Table \ref{Tab_steps}, for problem  \eqref{best_approx_prob} with $n=1000, m = 5$.}
\label{fig_best_approx_n1000}
\end{figure} 

\begin{remark}\label{remark_1_for_best_approx}
If we take $\hat{f}_k = \min \{f(x^1), \ldots, f(x^k)\}$ for all algorithms with steps in Table \ref{Tab_steps}, instead of the value of the objective function at the averaged points in each iteration, then we also can see that that Algorithm \ref{alg_mirror_descent} still works better. Although the work of other algorithms will improve and be better than the results in Figure \ref{fig_best_approx_n1000}. 

\begin{figure}[htp]
\minipage{0.95\textwidth}
\includegraphics[width=\linewidth]{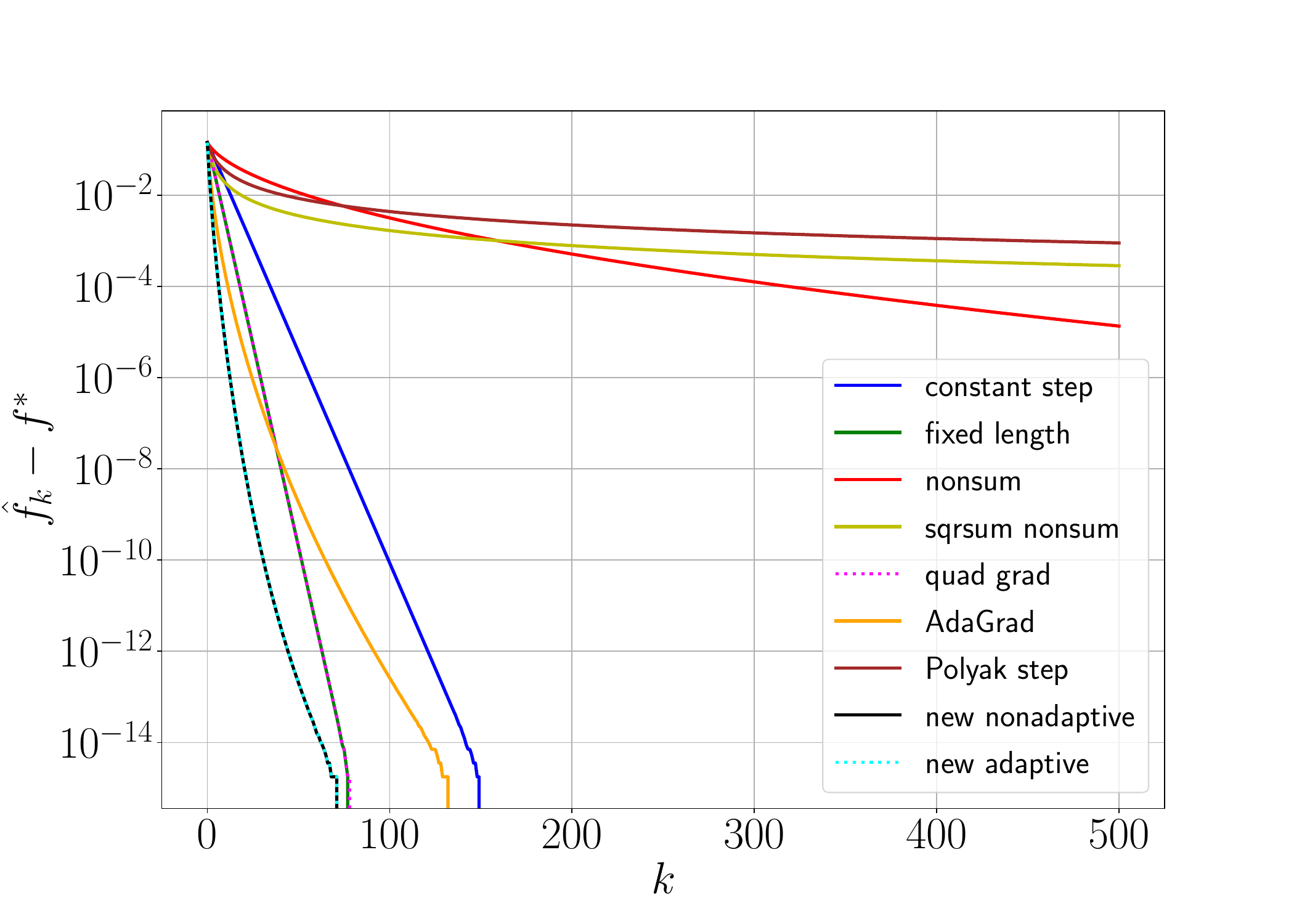}
\endminipage%\hfill
\caption{Results of  Algorithm \ref{alg_mirror_descent} and projected subgradient method using different step size rules listed in Table \ref{Tab_steps}, for problem  \eqref{best_approx_prob} with $n=1000, m = 5$. Here $\hat{f}_k = \min \{f(x^1), \ldots, f(x^k)\}$. }
\label{fig_best_approx_n1000_outs_mins}
\end{figure} 

\end{remark}

\begin{remark}\label{remark2_best_approx}
As we mentioned in Remark \ref{remark_k_gretear_minus1_mirror}, in comparison with the sub-optimal convergence rate \eqref{rate_mirror_k_minus1}, when $m > 0$, the weighting scheme $\frac{1}{\sum_{k= 1}^{N} \gamma_k^{-m}} \sum_{k = 1}^{N} \gamma_k^{-m} x^k$ (see Theorem \ref{theo_main_ineq_mirror_desc}) assigns smaller weights to the initial points and larger weights to the most recent points that generated by Algorithm \ref{alg_mirror_descent}. In Fig. \ref{fig_best_approx_n1000_differ_m}, we can see this fact, where when we increase the value of the parameter $m$, we can see that convergence of Algorithm \ref{alg_mirror_descent} and its performance becomes significantly better. 

\begin{figure}[htp]
\minipage{0.95\textwidth}
\includegraphics[width=\linewidth]{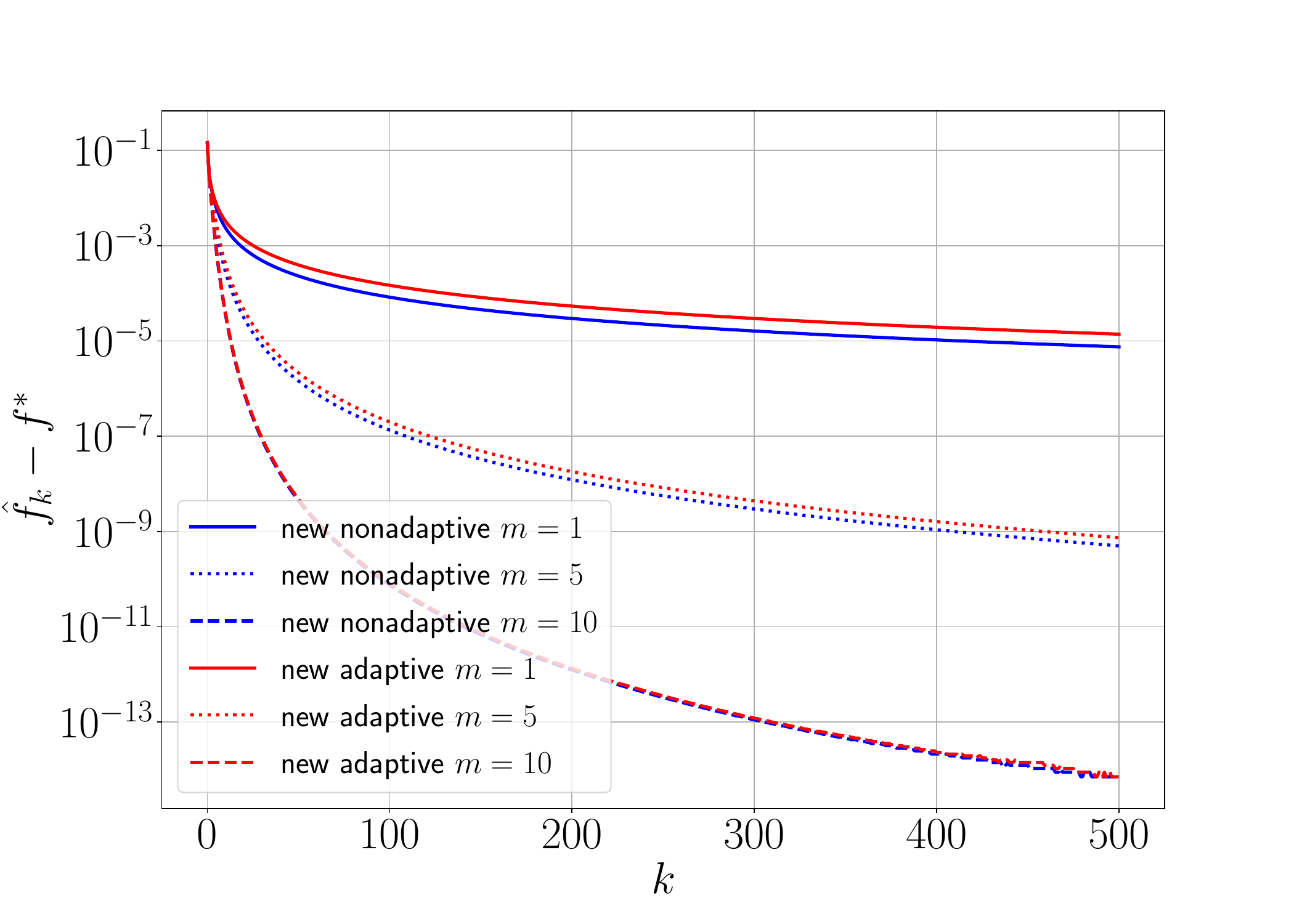}
\endminipage%\hfill
\caption{Results of  Algorithm \ref{alg_mirror_descent} for  problem  \eqref{best_approx_prob} with $n=1000$ and different values of the parameter $m$.}
\label{fig_best_approx_n1000_differ_m}
\end{figure}
\end{remark}

%%%%%%%%%%%%%%%%%%%%%%%%%%%%%%%%%%%%%%%%%%%%%%%
\subsection{Fermat–Torricelli–Steiner problem.}

Let $A_j \in \mathbb{R}^n, j = 1, \ldots, T$ be a given set of $T$ points, and let us consider an analogue of the well-known Fermat–Torricelli–Steiner problem. For this we need to solve the following optimization problem
\begin{equation}\label{Fermat_prob}
 \min_{x \in Q}\left\{ f(x) : = \frac{1}{T}\sum_{j = 1}^{T} \|x -  A_j  \|_2 \right\}. 
\end{equation}

The points $A_j, j = 1, 2, \ldots, T$ are randomly generated from a uniform distribution over $[0, 1)$. We run all algorithms (except the algorithm with Polyak step size since we cannot know the optimal value $f^*$ for problem \eqref{Fermat_prob}), with the same initial point $x^1 = \left(\frac{1}{\sqrt{n}}, \ldots, \frac{1}{\sqrt{n}}\right) \in Q$. 

The results of the comparison, for problem \eqref{Fermat_prob} with $n = 200$ and $T= 25$,   are presented in Fig. \ref{fig_Fermat_n200}. The reason here for taking $n = 200, T = 25$ is that we calculated the value $f_{\text{min}}$ using SciPy which does not work well for large values. In Fig. \ref{fig_Fermat_n200}, $\hat{f}_k$  denotes the value of the objective function $f$ at the averaged points in each iteration of all compared algorithms. From this figure, we can see that Algorithm \ref{alg_mirror_descent}, with non-adaptive and adaptive step size rules \eqref{steps_rules}, outperforms the other methods providing a better solution. % works better than other algorithms, where the difference between the performance of Algorithm \ref{alg_mirror_descent} and the rest of algorithms with steps in Table \ref{Tab_steps} is clear and significant. We also note that there is no significant difference in the performance of the Algorithm \ref{alg_mirror_descent} when using the adaptive or non-adaptive step size rules steps 

\begin{figure}[htp]
\minipage{0.95\textwidth}
\includegraphics[width=\linewidth]{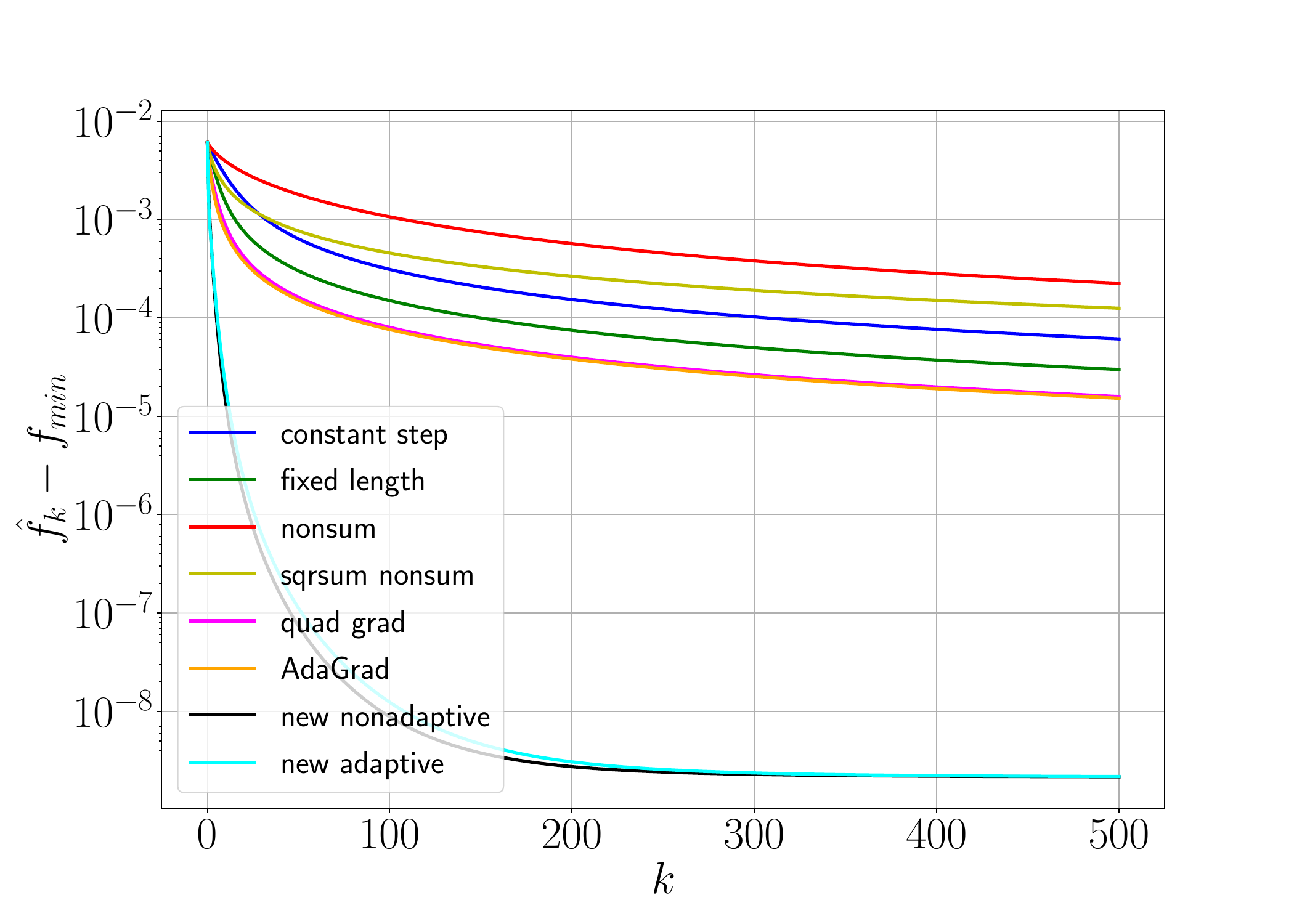}
\endminipage%\hfill
\caption{Results of  Algorithm \ref{alg_mirror_descent} and projected subgradient method using different step size rules listed in Table \ref{Tab_steps}, for problem \eqref{Fermat_prob} with $n=200, T = 25, m = 5$.}
\label{fig_Fermat_n200}
\end{figure} 

\begin{remark}
As in Remark \ref{remark_1_for_best_approx}. If we take $\hat{f}_k = \min \{f(x^1), \ldots, f(x^k)\}$ for all algorithms with steps in Table \ref{Tab_steps}, instead of the value of the objective function at the averaged points in each iteration, then we also can see that that Algorithm \ref{alg_mirror_descent} still works better at the first iterations and after that, it works the same as other algorithms (except algorithm with ''nonsum'' and ''sqrsum nonsum''). %Although the work of other algorithms will improve and be better than the results in Figure \ref{fig_Fermat_n200}.  

\begin{figure}[htp]
\minipage{0.95\textwidth}
\includegraphics[width=\linewidth]{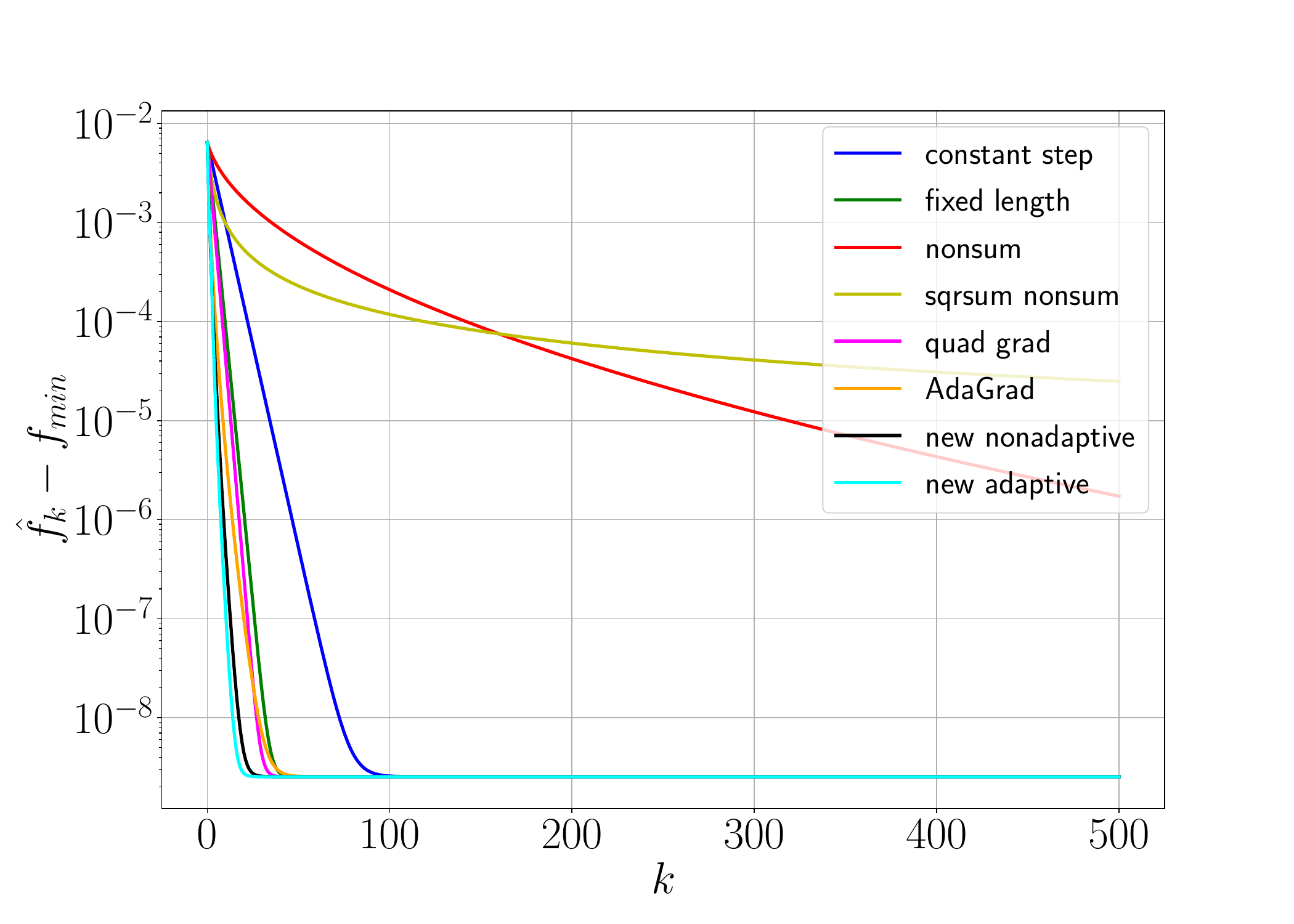}
\endminipage%\hfill
\caption{Results of  Algorithm \ref{alg_mirror_descent} and projected subgradient method using different step size rules listed in Table \ref{Tab_steps}, for problem  \eqref{Fermat_prob} with $n=200, T = 25, m = 5$. Here $\hat{f}_k = \min \{f(x^1), \ldots, f(x^k)\}$. }
\label{fig_Fermat_n200_outs_mins}
\end{figure} 

\end{remark}

\begin{remark}
As in Remark \ref{remark2_best_approx}. We can see that when we increase the value of the parameter $m$ in the weighting scheme $\frac{1}{\sum_{k= 1}^{N} \gamma_k^{-m}} \sum_{k = 1}^{N} \gamma_k^{-m} x^k$ (see Theorem \ref{theo_main_ineq_mirror_desc})  the convergence of Algorithm \ref{alg_mirror_descent} and its performance becomes better.

\begin{figure}[htp]
\minipage{0.95\textwidth}
\includegraphics[width=\linewidth]{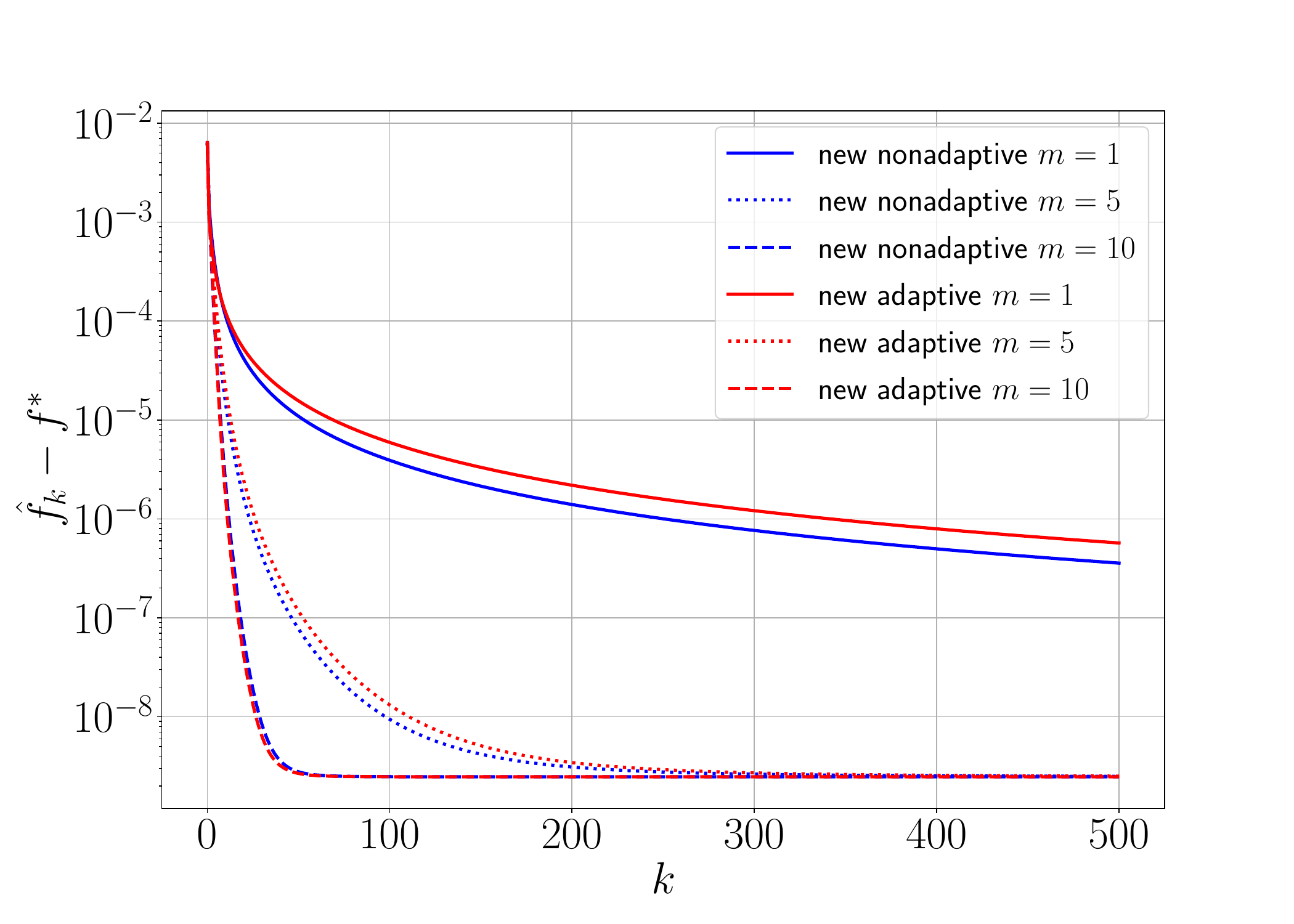}
\endminipage%\hfill
\caption{Results of  Algorithm \ref{alg_mirror_descent} for  problem  \eqref{Fermat_prob} with $n=200, T = 25$ and different values of the parameter $m$.}
\label{fig_Fermat_n200_differ_m}
\end{figure}
\end{remark}

\begin{remark}
Because of SciPy does not work well for large values of $n$ and $T$, we conduct experiments without calculating $f_{\text{min}}$. We run algorithms with $n = 1000$ and $T = 100$. The results are presented in Fig. \ref{fig_Fermat_n1000}. In this figure, we show the dynamics of $\hat{f}_k$ (the value of the objective function at the averaged points in each iteration) as a function of $k$. From Fig. \ref{fig_Fermat_n1000}, as in the previous results, we also see that Algorithm \ref{alg_mirror_descent}, with non-adaptive and adaptive step size rules \eqref{steps_rules}, outperforms the other methods providing a better solution. Note that in the performance of the Algorithm \ref{alg_mirror_descent}, 
 there is a difference if we choose adaptive or non-adaptive step sizes.  

\begin{figure}[htp]
\minipage{0.95\textwidth}
\includegraphics[width=\linewidth]{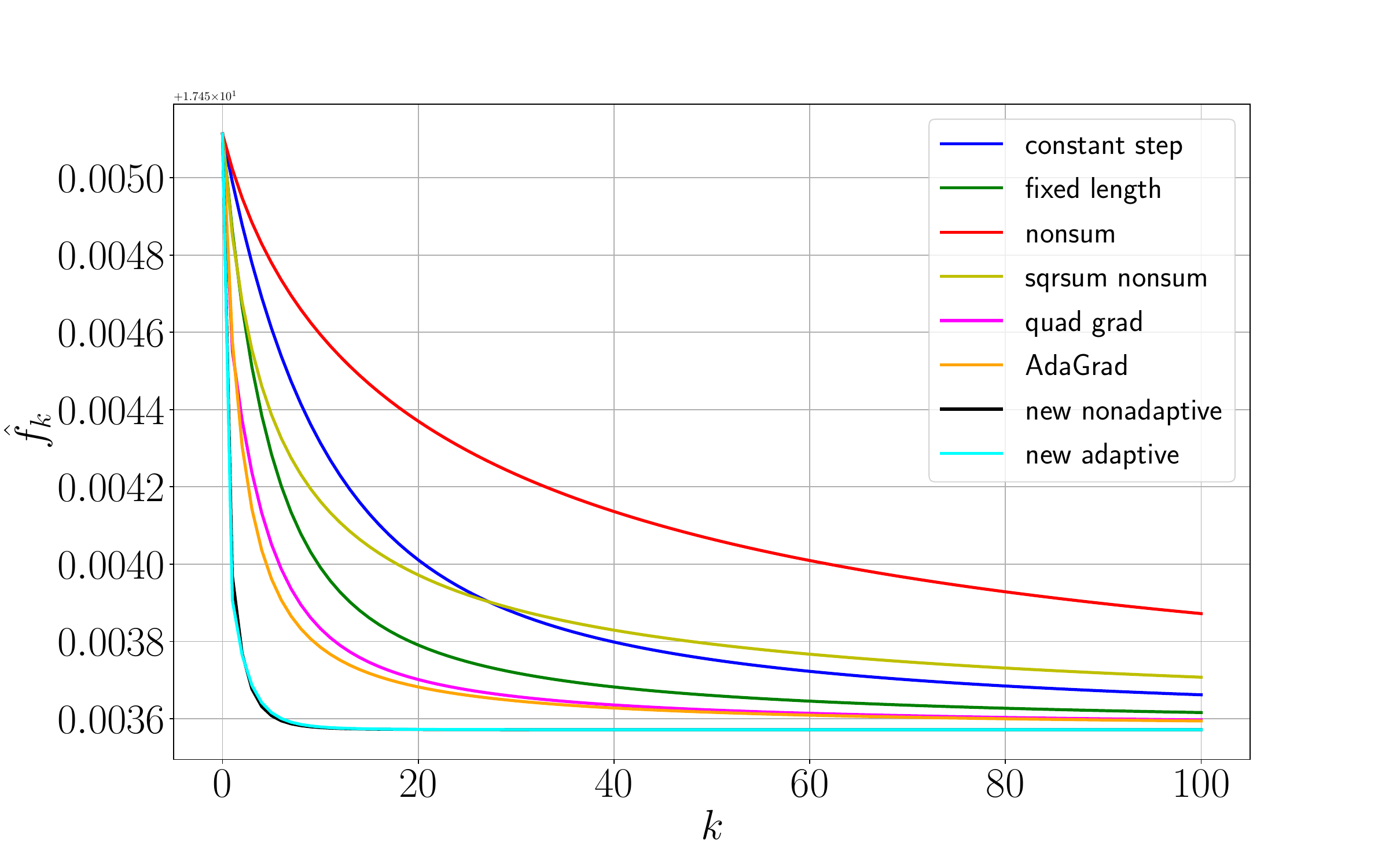}
\endminipage%\hfill
\caption{Results of  Algorithm \ref{alg_mirror_descent} and projected subgradient method using different step size rules listed in Table \ref{Tab_steps}, for problem  \eqref{Fermat_prob} with $n=1000, T = 100, m = 5$.}
\label{fig_Fermat_n1000}
\end{figure}
\end{remark}

%%%%%%%%%%%%%%%%%%%%%%%%%%%%%%%%%%%%%%%%%%%%%
\subsection{Smallest covering ball problem}

Let $A_j \in \mathbb{R}^n, j = 1, \ldots, T$ be a given set of $T$ points, and let us consider an analogue of the well-known smallest covering
ball problem. For this we need to solve the following optimization problem
\begin{equation}\label{smallest_covering_prob}
\min_{x \in Q}\left\{ f(x) : = \max_{1 \leq j \leq T} \|x -  A_j  \|_2 \right\}. 
\end{equation}
The points $A_j, j = 1, 2, \ldots, T$ are randomly generated from a uniform distribution over $[0, 1)$. We run all algorithms (except the algorithm with Polyak step size since we cannot know the optimal value $f^*$ for problem \eqref{Fermat_prob}), with the same initial point $x^1 = \left(\frac{1}{\sqrt{n}}, \ldots, \frac{1}{\sqrt{n}}\right) \in Q$. 

The results of the comparison, for problem \eqref{smallest_covering_prob} with $n = 200$ and $T= 25$, are presented in Fig. \ref{fig_Fermat_n200} (in the left, where we show the dynamics of $\hat{f}_k - f_{\text{min}}$ as a function of $k$, and $f_{\text{min}}$ calculated by SciPy), and with $n = 1000$ and $T= 100$ (in the right, where we show the dynamics of $\hat{f}_k$ as a function of $k$).  

From Fig. \ref{fig_smallest_covering_n200}, we also see that Algorithm \ref{alg_mirror_descent}, with non-adaptive and adaptive step size rules \eqref{steps_rules}, outperforms the other methods providing a better solution, without any difference if we choose adaptive or non-adaptive step sizes.  

\begin{figure}[htp]
\minipage{0.48\textwidth}
\includegraphics[width=\linewidth]{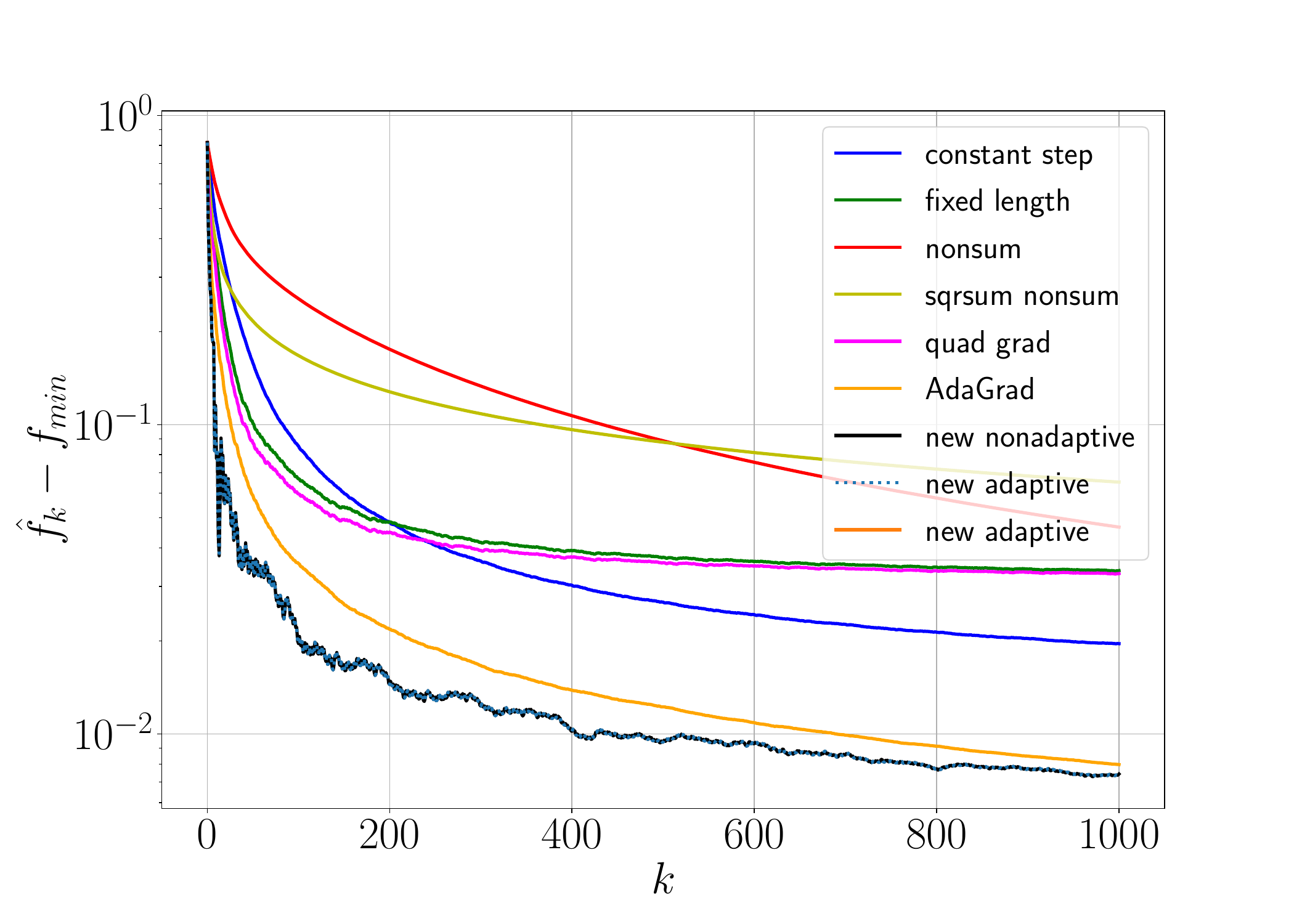}
\endminipage%\hfill
\minipage{0.55\textwidth}
\includegraphics[width=\linewidth]{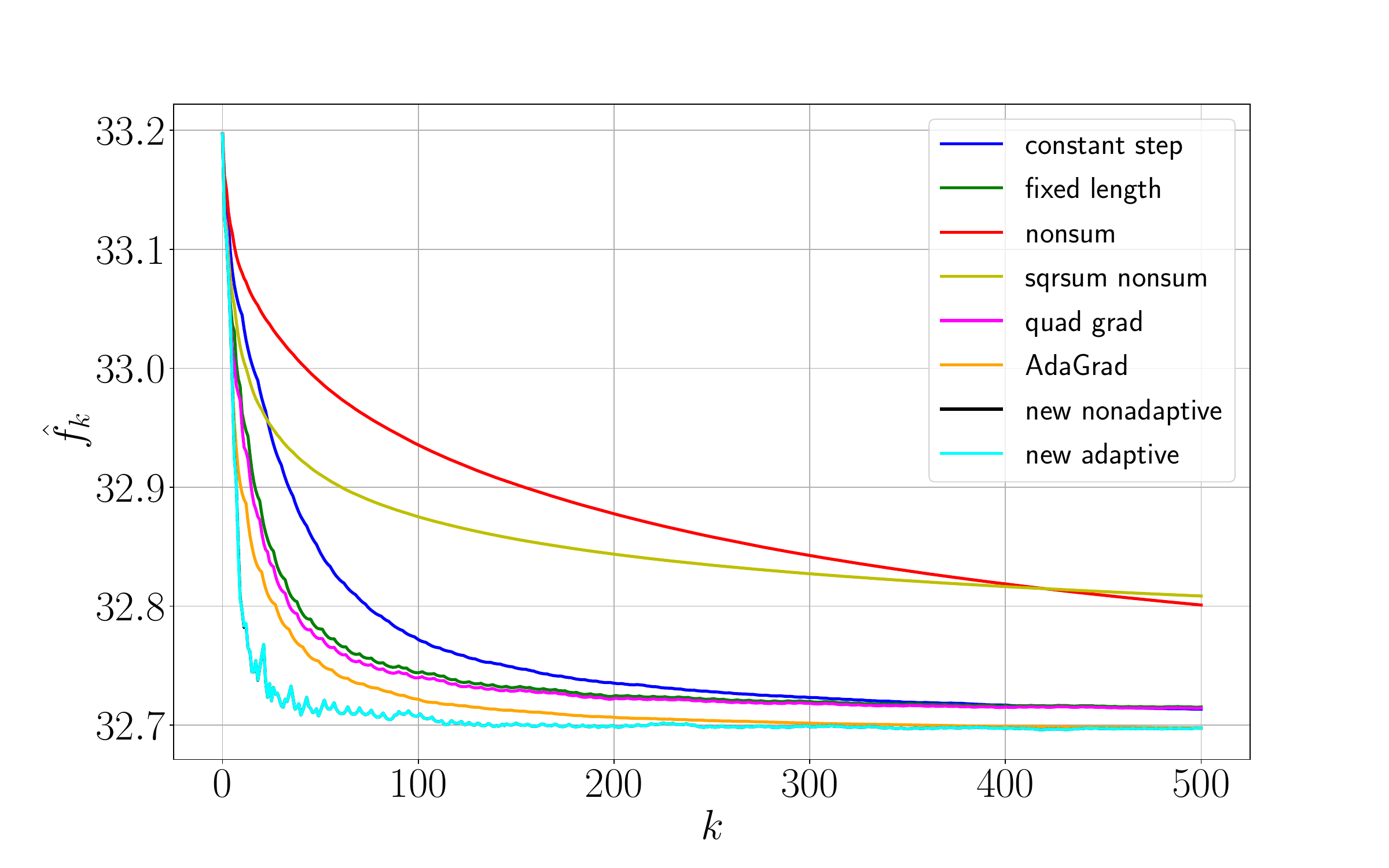}
\endminipage%\hfill
\caption{Results of  Algorithm \ref{alg_mirror_descent} and projected subgradient method using different step size rules listed in Table \ref{Tab_steps}, for problem  \eqref{smallest_covering_prob} with $n=200, T = 25, m = 5$ (left). And with $n = 1000, T = 100, m = 5$ (right).}
\label{fig_smallest_covering_n200}
\end{figure}

%%%%%%%%%%%%%%%%%%%%%%%%%%%%%%%%%%%%%%%%%%%%%%%%%
\subsection{Maximum of a finite collection of linear functions}

In this subsection, we consider the minimization problem of a function that represents a point-wise maximum of a finite collection of linear functions. This problem has the following form:
\begin{equation}\label{prob_max_linears}
\min_{x \in Q} \left[ f(x): = \max \left\{ f_i (x) = \langle a_i, x \rangle + b_i, i = 1, 2, \ldots, T\right\} \right], 
\end{equation}
where $a_i \in \mathbb{R}^n$, and $b_i \in \mathbb{R}$ ($\forall i = 1, \ldots, T$). 

The vectors $a_j$ and constants $b_j$, for $ j = 1, 2, \ldots, T$,  are randomly generated from a uniform distribution over $[0, 1)$. We run all algorithms (except the algorithm with Polyak step size since we cannot know the optimal value $f^*$ for problem \eqref{Fermat_prob}), with the same initial point $x^1 = \left(\frac{1}{\sqrt{n}}, \ldots, \frac{1}{\sqrt{n}}\right) \in Q$. 

The results of the comparison, for problem \eqref{smallest_covering_prob} with $n = 200$ and $T= 25$, are presented in Fig. \ref{fig_max_linears_n200} (in the left, where we show the dynamics of $\hat{f}_k - f_{\text{min}}$ as a function of $k$, and $f_{\text{min}}$ calculated by SciPy), and with $n = 1000$ and $T= 100$ (in the right, where we show the dynamics of $\hat{f}_k$ as a function of $k$).

From Fig. \ref{fig_max_linears_n200}, we also see that Algorithm \ref{alg_mirror_descent}, with non-adaptive and adaptive step size rules \eqref{steps_rules}, outperforms the other methods providing a better solution, without any difference if we choose adaptive or non-adaptive step sizes.  

\begin{figure}[htp]
\minipage{0.48\textwidth}
\includegraphics[width=\linewidth]{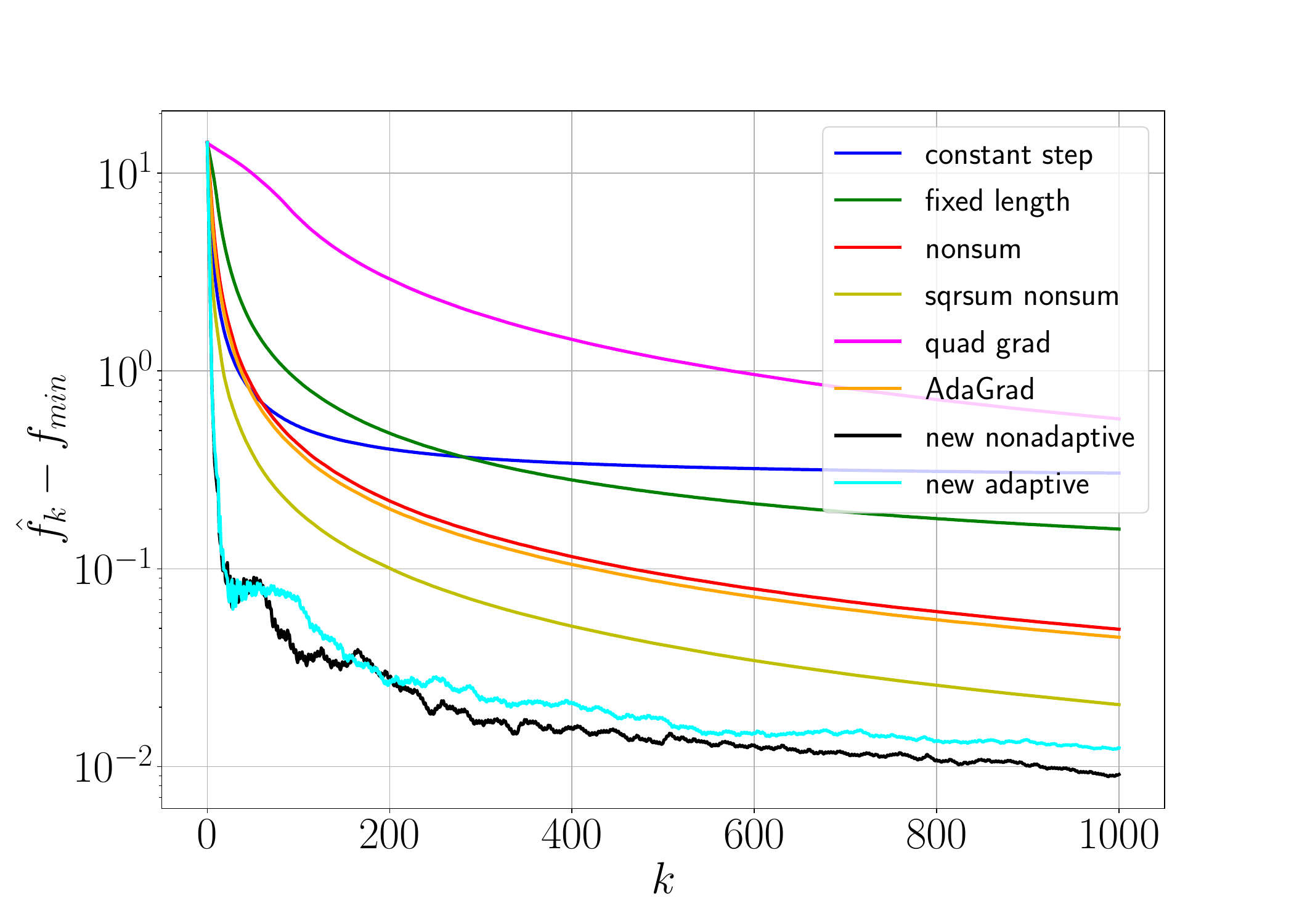}
\endminipage%\hfill
\minipage{0.55\textwidth}
\includegraphics[width=\linewidth]{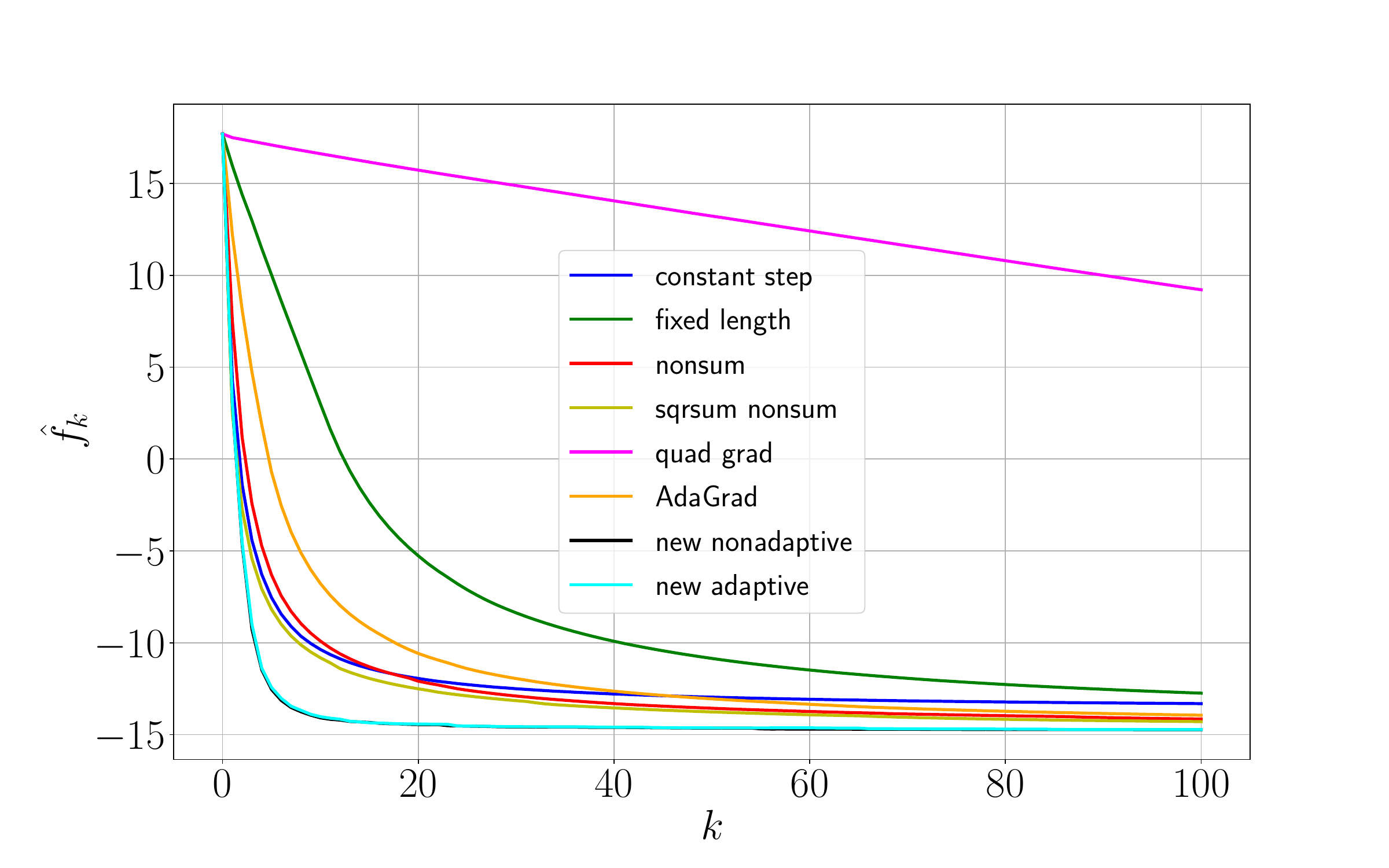}
\endminipage%\hfill
\caption{Results of  Algorithm \ref{alg_mirror_descent} and projected subgradient method using different step size rules listed in Table \ref{Tab_steps}, for problem  \eqref{prob_max_linears} with $n=200, T = 25, m = 5$ (left). And with $n = 1000, T = 100, m = 5$ (right).}
\label{fig_max_linears_n200}
\end{figure} 

%%%%%%%%%%%%%%%%%%%%%%%%%%%%%%%%%%%%%%%%%%%%%%%%%
{\color{black}
\subsection{Problems with functional constraints}
In this subsection, we will show the advantages of the proposed Algorithm \ref{alg_new_constraints} , and compare its work with algorithms 1 and 2 from  \cite{article:adaptive_mirror_2018}, which were proposed to solve the problem \eqref{general_problem_1}. 

First, we mention that for the following adaptive step size scheme
\begin{equation}\label{adaptivestepsalg3}
\gamma_k = 
\begin{cases}
\gamma_k^f : = \frac{\sqrt{2 \sigma}}{\|\nabla f(x^k)\|_* \sqrt{k}}, \quad  \text{if} \; k \in I, \\
\gamma_k^g : = \frac{\sqrt{2 \sigma}}{ \|\nabla g(x^k)\|\sqrt{k}}, \quad  \text{if} \; k \in J.
\end{cases}
\end{equation}
there is no guarantee that the sequence $\{\gamma_k\}_{k \geq 1}$ is non-increasing. As a result, this step size does not satisfy one of the conditions of Theorem \ref{theorem_alg3}. However, Algorithm \ref{alg_new_constraints} performs well in practice, and the convergence rate remains the same as the step size defined in \eqref{steps_rules}.  Therefore, we need to consider restructuring the scheme of step sizes so that it becomes adaptive and allows us to obtain the same results as previously obtained for the non-adaptive ones.

We will consider the problems \eqref{best_approx_prob}, \eqref{Fermat_prob}, \eqref{smallest_covering_prob} and \eqref{prob_max_linears}, with the following functional constraint
\begin{equation}\label{functional_const}
g(x) = \max_{}\left\{ g_i(x) = \langle \alpha_i, x \rangle - \beta_i, \quad i = 1, \ldots, p \right\},   
\end{equation}
where $\alpha_i \in \mathbb{R}^n$ and $\beta_i \in \mathbb{R}$ are randomly generated from a uniform distribution over $[0, 1)$. 

As in the previous, we take $Q$ as a unit ball in $\mathbb{R}^n$ with the center at $0 \in \mathbb{R}^n$. 
All compared methods start from the same initial point $x^1 = 0 \in Q$. We take $n=1000$, $p = 100$ functional constraints in \eqref{functional_const} and $\varepsilon = 10^{-3}$.  

The comparison results are presented in Fig. \ref{fig_bestapprox_fermat} and Fig. \ref{fig_smallestcovering_maxlinears}. In these figures, we show the objective function values in each problem at the points corresponding to the set of productive steps generated by the compared algorithms.  

From Figures \ref{fig_bestapprox_fermat} and \ref{fig_smallestcovering_maxlinears}, we can see how the proposed Algorithm \ref{alg_new_constraints} is the best, and we note that the work of this algorithm will be better if we increase the value of the parameter $m$.

\begin{figure}[htp]
\centering
\minipage{0.52\textwidth}
\includegraphics[width=\linewidth]{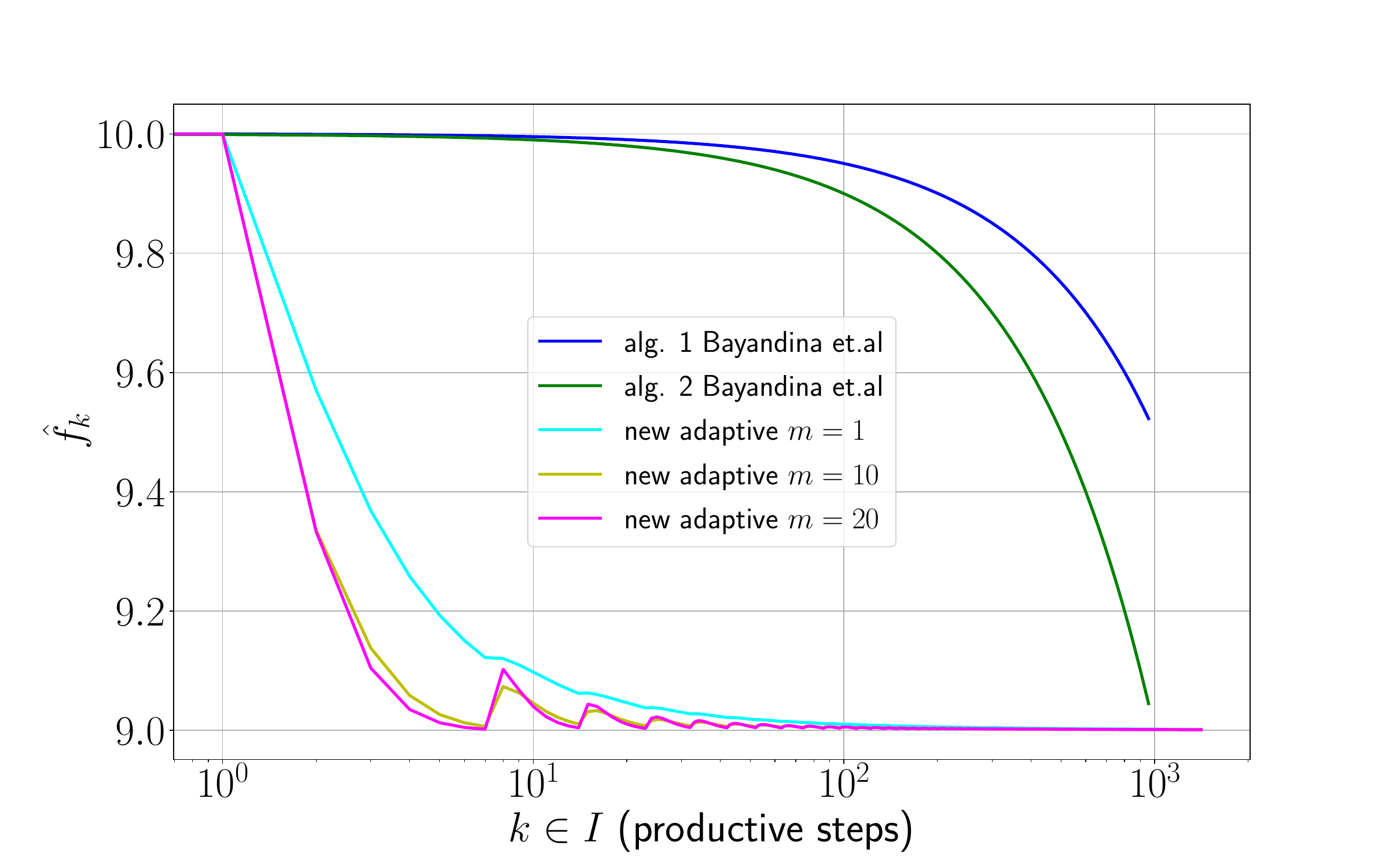}
\endminipage%\hfill
\minipage{0.52\textwidth}
\includegraphics[width=\linewidth]{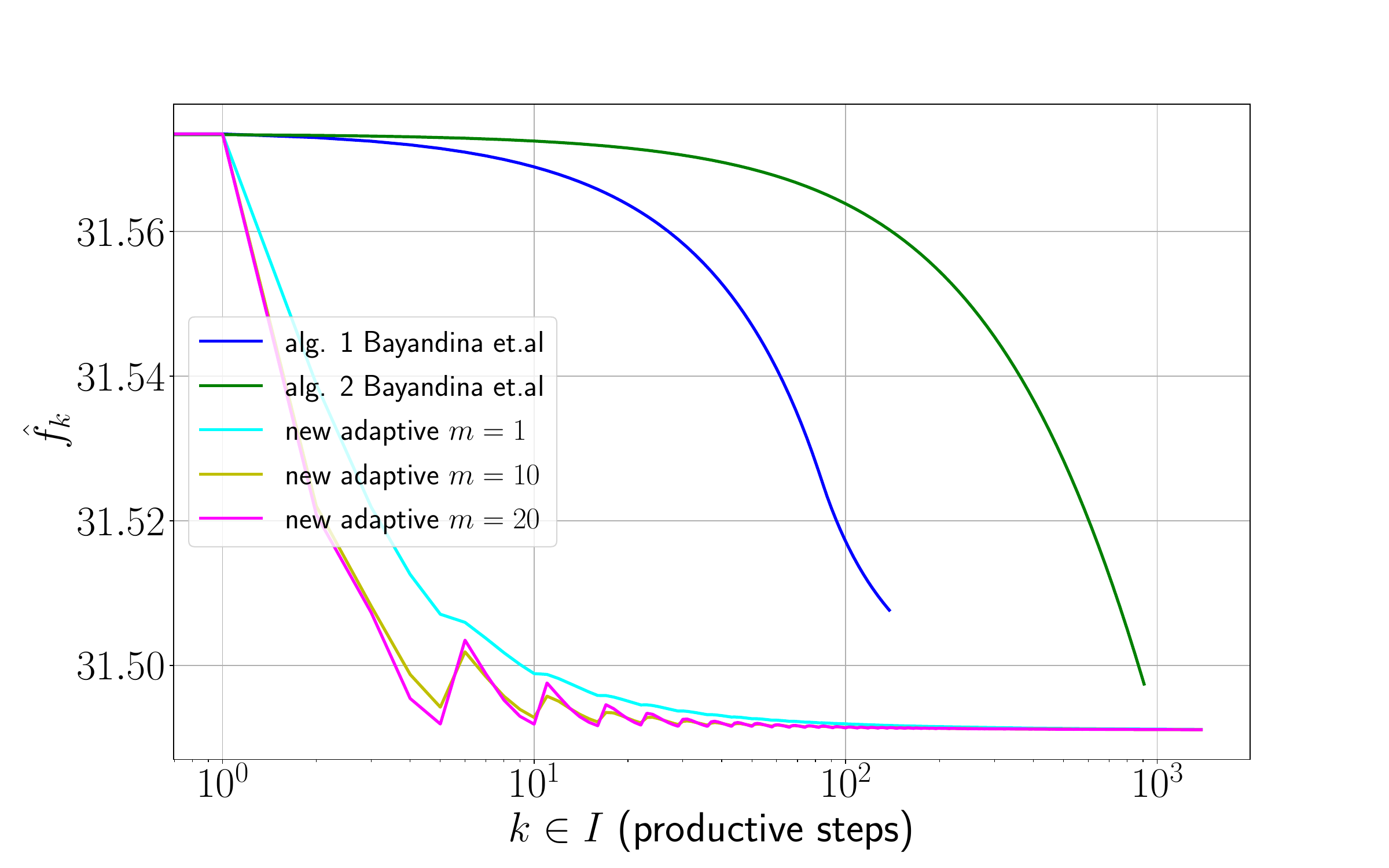}
\endminipage%\hfill
\caption{Results of  Algorithm \ref{alg_new_constraints} and Algorithms 1, 2 in \cite{article:adaptive_mirror_2018}, for  problem  \eqref{best_approx_prob} (left) and problem \eqref{Fermat_prob} (right), with functional constraints \eqref{functional_const}, and  $n=1000, p = 100, T = 100$.}
\label{fig_bestapprox_fermat}
\end{figure} 

\begin{figure}[htp]
\centering
\minipage{0.52\textwidth}
\includegraphics[width=\linewidth]{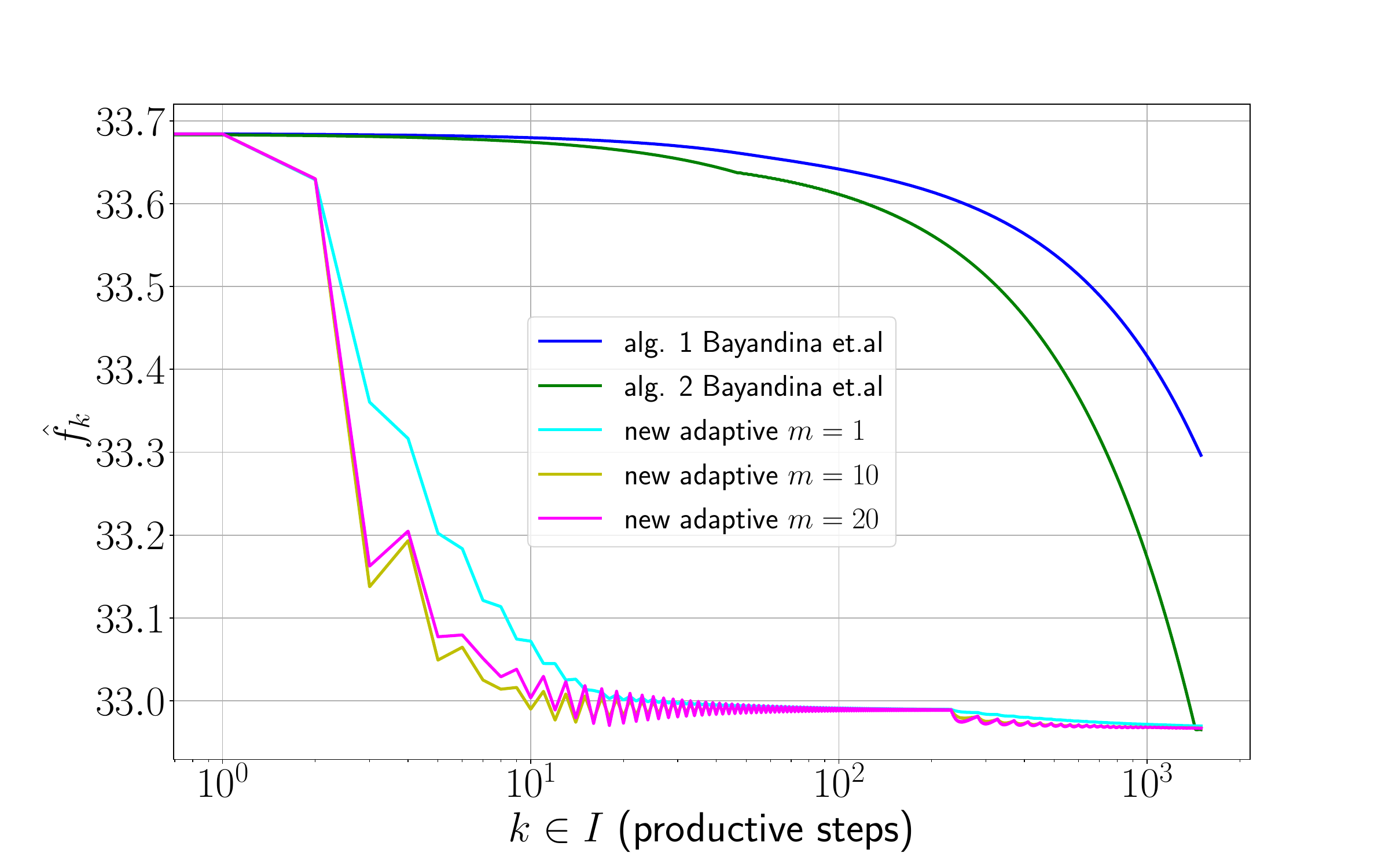}
\endminipage%\hfill
\minipage{0.52\textwidth}
\includegraphics[width=\linewidth]{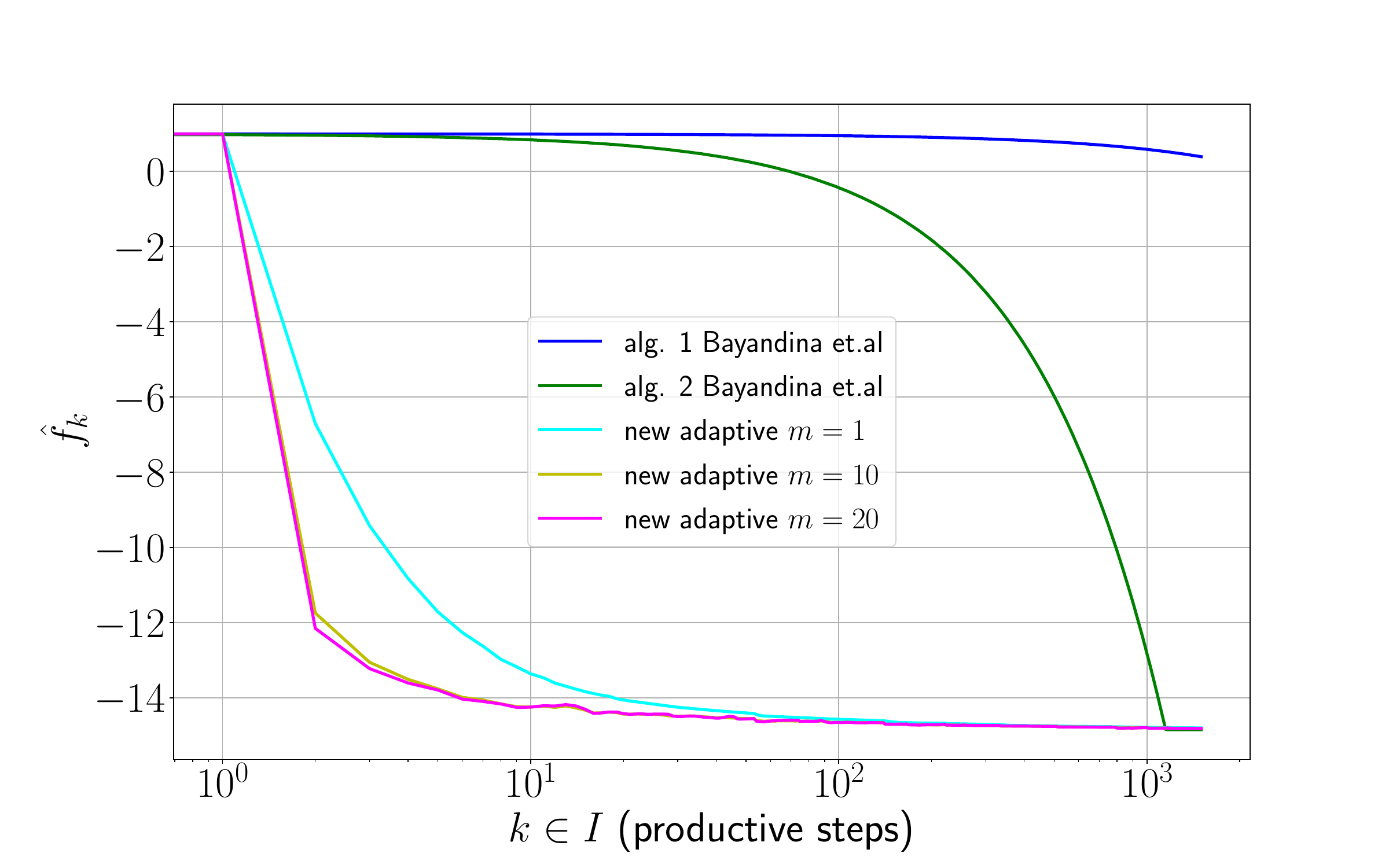}
\endminipage%\hfill
\caption{Results of  Algorithm \ref{alg_new_constraints} and Algorithms 1, 2 in \cite{article:adaptive_mirror_2018}, for  problem  \eqref{smallest_covering_prob} (left) and problem \eqref{prob_max_linears} (right), with functional constraints \eqref{functional_const}, and $n=1000, p = 100, T = 100$.}
\label{fig_smallestcovering_maxlinears}
\end{figure} 
}

Now, let us compare Algorithm \ref{alg_new_constraints} (with stopping rule \eqref{stop_criter_alg3} and adaptive variant step sizes \eqref{adaptivestepsalg3}) with its modified version (Algorithm \ref{alg_new_constraints_modif}).

We will consider the problems \eqref{best_approx_prob}, \eqref{Fermat_prob}, \eqref{smallest_covering_prob} and \eqref{prob_max_linears}, with the following functional constraint
\begin{equation}\label{functional_const_alg4}
g_i(x) = \langle \alpha_i, x \rangle - \beta_i, \quad i = 1, \ldots, p ,   
\end{equation}
where $\alpha_i \in \mathbb{R}^n$ and $\beta_i \in \mathbb{R}$ are randomly generated from the normal (Gaussian) distribution with mean equals $0$ and standard deviation equals $1$. 

The comparison results are presented in Tables \ref{Tab1} and \ref{Tab2}. In these tables, we show the running time of Algorithms \ref{alg_new_constraints} and \ref{alg_new_constraints_modif} for different values of $\varepsilon$.

\begin{table}[]
\begin{tabular}{|c|ccccc|}
\hline
\multirow{2}{*}{$\varepsilon$} & \multicolumn{1}{c|}{$1/2$} & \multicolumn{1}{c|}{$1/4$} & \multicolumn{1}{c|}{$1/8$} & \multicolumn{1}{c|}{$1/16$} & $1/32$ \\ \cline{2-6} 
& \multicolumn{5}{c|}{$ m = 2$} \\ \hline
Algorithm \ref{alg_new_constraints}& \multicolumn{1}{c|}{4.902043} & \multicolumn{1}{c|}{18.470622} & \multicolumn{1}{c|}{69.866258} & \multicolumn{1}{c|}{262.627478} & 971.659847 \\ \hline
Algorithm \ref{alg_new_constraints_modif}& \multicolumn{1}{c|}{3.057222} & \multicolumn{1}{c|}{10.764352} & \multicolumn{1}{c|}{42.096594} & \multicolumn{1}{c|}{170.469682} & 667.055894 \\ \hline\hline
& \multicolumn{5}{c|}{ $m = 5$}    \\ \hline
Algorithm \ref{alg_new_constraints} & \multicolumn{1}{c|}{16.870511} & \multicolumn{1}{c|}{55.661553} & \multicolumn{1}{c|}{226.497831} & \multicolumn{1}{c|}{859.8155} & 3475.06587 \\ \hline
Algorithm \ref{alg_new_constraints_modif}& \multicolumn{1}{c|}{12.761174} & \multicolumn{1}{c|}{40.493333} & \multicolumn{1}{c|}{150.768815} & \multicolumn{1}{c|}{538.384425} & 2000.461263 \\ \hline
\end{tabular}
\caption{The running time (in seconds) of Algorithms \ref{alg_new_constraints} and \ref{alg_new_constraints_modif} for problem \eqref{best_approx_prob} with functional constraints \eqref{functional_const_alg4} and $n = 100, p = 50.$  }
\label{Tab1}
\end{table}

\begin{table}[]
\begin{tabular}{|c|ccccc|}
\hline
\multirow{2}{*}{$\varepsilon$} & \multicolumn{1}{c|}{$1/2$} & \multicolumn{1}{c|}{$1/4$} & \multicolumn{1}{c|}{$1/8$} & \multicolumn{1}{c|}{$1/16$} & $1/32$ \\ \cline{2-6} 
& \multicolumn{5}{c|}{$ m = 2$} \\ \hline
Algorithm \ref{alg_new_constraints}& \multicolumn{1}{c|}{2.786089} & \multicolumn{1}{c|}{11.755107} & \multicolumn{1}{c|}{46.759482} & \multicolumn{1}{c|}{191.072985} & 766.52804 \\ \hline
Algorithm \ref{alg_new_constraints_modif}& \multicolumn{1}{c|}{2.984926} & \multicolumn{1}{c|}{8.90336} & \multicolumn{1}{c|}{35.072455} & \multicolumn{1}{c|}{140.398595} & 559.14492 \\ \hline\hline
& \multicolumn{5}{c|}{ $m = 5$}    \\ \hline
Algorithm \ref{alg_new_constraints} & \multicolumn{1}{c|}{11.560235} & \multicolumn{1}{c|}{46.171449} & \multicolumn{1}{c|}{188.252824} & \multicolumn{1}{c|}{762.623349} &  3008.319499\\ \hline
Algorithm \ref{alg_new_constraints_modif}& \multicolumn{1}{c|}{9.813089} & \multicolumn{1}{c|}{35.101363} & \multicolumn{1}{c|}{140.252735} & \multicolumn{1}{c|}{552.735558} & 2327.445081 \\ \hline
\end{tabular}
\caption{The running time (in seconds) of Algorithms \ref{alg_new_constraints} and \ref{alg_new_constraints_modif} for problem \eqref{prob_max_linears} with functional constraints \eqref{functional_const_alg4} and $n = 100, p = 50, T = 50.$  }
\label{Tab2}
\end{table}

From Tables \ref{Tab1} and \ref{Tab2}, we see that Algorithm \ref{alg_new_constraints_modif} can significantly reduce the running time to achieve the desired accuracy of a solution of the problems under consideration. 

\bigskip
\noindent
\textbf{Acknowledgment.}  We would like to thank Bowen Yuan for pointing out the inaccuracy in the theoretical result presented in the previous version, specifically regarding the adaptive step size scheme and the fact that the associated sequence is not necessarily non-increasing.

%%%%%%%%%%%%%%%%%%%%%%%%%%%%%%%%%%%%%%%%%%%

%\begin{figure}[h]
%	\centering
%	\subfloat[My simulationrun on Windows]{{\includegraphics[width=6.5cm]{single_sine_perturb} }}%
%	\qquad
%	\subfloat[My simulation run on Linux]{{\includegraphics[width=6.5cm]{single_sine_perturb} }}%
%	\caption{Two figure example 1}%
%	\label{fig:Two figure}%
%\end{figure}

%\begin{figure}[H]
%	\centering
%	\begin{subfigure}{0.49\linewidth} \centering
%		\includegraphics[scale=0.25]{Figures/single_sine_perturb}
%		\caption{S.C.L with sinusoidal disturbance}\label{fig:figA}
%	\end{subfigure}
%	\begin{subfigure}{0.49\linewidth} \centering
%		\includegraphics[scale=0.25]{Figures/dual_sine_perturb}
%		\caption{D.C.L with sinusoidal disturbance}\label{fig:figB}
%	\end{subfigure}
%	\caption{Two figure example 2} \label{fig:twodisturb}
%\end{figure}

\end{document}